\documentclass[12pt]{amsart}
\usepackage{geometry}   %????????
\usepackage[colorlinks,citecolor = red, linkcolor=blue,hyperindex]{hyperref}
\usepackage{euscript,eufrak,verbatim, mathrsfs}
\usepackage[psamsfonts]{amssymb}
\usepackage{bbm}
\usepackage{graphicx}
 \usepackage{float}
\usepackage{float, tikz}
\usepackage{pgfplots}
\pgfplotsset{compat=1.18}
\usepgfplotslibrary{fillbetween}
\usepackage{caption}
\usepackage{subcaption}

 \usepackage{extarrows}
 \usepackage[all, cmtip]{xy}

\usepackage{upref, xcolor, dsfont}
\usepackage{amsfonts,amsmath,amstext,amsbsy, amsopn,amsthm}
\usepackage{enumerate}

\usepackage{url}

\usepackage{mathtools}
\usepackage{bookmark}

 \usepackage{euscript}
\usepackage{helvet}         % selects\textbf{\textbf{•}} Helvetica as sans-serif font
\usepackage{courier}        % selects Courier as typewriter font
\usepackage{type1cm}        % activate if the above 3 fonts are
\usepackage{multicol}        % used for the two-column index
\usepackage[bottom]{footmisc}% places footnotes at page bottom

\newtheorem{theorem}{Theorem}[section]
\newtheorem*{theorem*}{Theorem A}
\newtheorem{lemma}[theorem]{Lemma}

\newtheorem{proposition}[theorem]{Proposition}
\newtheorem{corollary}[theorem]{Corollary}
\newtheorem*{definition*}{Definition}
\newtheorem*{remark*}{Remark}

\newtheorem*{observation*}{Observation}

\newtheorem*{assumption*}{Assumption}

\newtheorem{conjecture}{Conjecture}
\newtheorem{problem}[conjecture]{Problem}
\theoremstyle{definition}
\newtheorem{definition}{Definition}[section]
\newtheorem{question}{Question}
\newtheorem*{problem*}{Problem}

\theoremstyle{remark}
\newtheorem{remark}{Remark}[section]
\newtheorem{claim}{Claim}[section]

\newcommand{\N}{\mathbb{N}}

\newcommand{\D}{\mathbb{D}}

\begin{document}

\title[Shimorin-type kernels]{$L^p$--$L^q$ estimates for Shimorin-type integral operators}

\author{Yuerang Li}
\address{Yuerang LI: College of Mathematics and Statistics, Chongqing University, Chongqing
401331, China}
\email{yuerangli@outlook.com}

\author{Zipeng Wang}
\address{Zipeng WANG: College of Mathematics and Statistics, Chongqing University, Chongqing
401331, China}
\email{zipengwang2012@gmail.com; zipengwang@cqu.edu.cn}

\author{Kenan Zhang}
\address{Kenan ZHANG, School of Mathematical Sciences, Fudan University, Shanghai, 200433,
China}
\email{knzhang21@m.fudan.edu.cn}


\maketitle

\begin{abstract}
Let $\nu$ be a positive measure on $[0,1]$. A Shimorin-type operator $T_\nu$ is an integral operator on the unit disk given by  
\[
T_\nu f(z) = \int_{\mathbb{D}} \frac{1}{1 - z\overline{\lambda}} \left( \int_0^1 \frac{d\nu(r)}{1 - r z\overline{\lambda}} \right) f(\lambda)\,dA(\lambda),
\]
which originates from Shimorin’s work on Bergman-type kernel representations for logarithmically subharmonic weighted Bergman spaces.  

In this paper, we study $L^p$–$L^q$ estimates for $T_\nu$.  
Unlike classical Bergman-type operators, the critical line on the $(1/p,1/q)$-plane that separates the boundedness and unboundedness regions of $T_\nu$ is not immediately evident. Moreover, even along this line, new phenomena arise. In the present work, by introducing a quantity $c_\nu$,

\begin{itemize}
    \item we first determine the critical boundary in the $(1/p,1/q)$-plane for bounded $T_\nu$;  
    \item furthermore, on this critical line, we establish necessary and sufficient conditions for $T_\nu$ which have standard Bergman-type $L^p$–$L^q$ estimates, meaning that it is bounded in the interior of the region and admits weak-type and BMO-type estimates at endpoints.
\end{itemize}
\end{abstract}

\setcounter{tocdepth}{1}

\tableofcontents

\setcounter{equation}{0}

\section{Introduction}
The Bergman-type operator on the unit disk $\mathbb{D}$ of the complex plane $\mathbb{C}$
\begin{align*}
K_\alpha f(z)=\int_{\mathbb{D}}\frac{f(\lambda)}{(1-z\overline{\lambda})^\alpha}dA(\lambda)
\end{align*}
is recognized as a complex Riesz potential on the unit disk, and the case $\alpha=2$ corresponds to the Bergman projection
\begin{align*}
K_2f(z)=Pf(z)=\int_{\mathbb{D}}\frac{f(\lambda)}{(1-z\overline{\lambda})^2}dA(\lambda)
\end{align*}
which serves as a toy example of a Calderón-Zygmund operator on spaces of homogeneous type (see, e.g., \cite{CRW}, \cite{DHZZ}). These operators are among the most extensively studied objects in the operator theory of analytic function spaces (see, e.g., \cite{lue87}, \cite{Zhu14}).

Let \( \nu \) be a finite positive Borel measure on \( [0,1] \). In this paper, we study the $L^p$–$L^q$ estimates for the Shimorin-type integral operator
\begin{align*}
    T_\nu f(z) = \int_{\mathbb{D}} \frac{1}{1 - z\overline{\lambda}} \left( \int_0^1 \frac{d\nu(r)}{1 - r z\overline{\lambda}} \right) f(\lambda)\,dA(\lambda).
\end{align*}
When \( \nu = \delta_0 \), \( T_\nu \) reduces to \( K_1 \) induced by the Hardy kernel; when \( \nu = \delta_1 \), it recovers the Bergman projection \( P \). Here, \( dA \) denotes the normalized Lebesgue measure on $\mathbb{D}$, and \( \delta_x \) (for \( x \in \{0,1\} \)) is the Dirac measure concentrated at \( x \). Unlike Bergman-type operators (see Figures \ref{alpha2} and \ref{alpha3}), the region of boundedness for Shimorin-type operators lacks a clear-cut boundary. Moreover, even along the critical line, new phenomena still arise.

The integral kernel of $T_\nu$ is denoted by
\begin{align}\label{sio-kernel}
K_{\nu}(z,\lambda) = \frac{1}{1 - z\overline{\lambda}}\int_0^1 \frac{1}{1 - rz\overline{\lambda}} \,d\nu(r).
\end{align}
Established by Shimorin \cite{Shi01}, this kernel is related to Bergman-type kernel representations of weighted Bergman spaces induced by logarithmically subharmonic weights on the unit disk \cite[Problem~19]{AHR}. As noted in \cite{AHR}, the problem was suggested by Shimorin. The representation \eqref{sio-kernel} gives an affirmative answer for radial weights, while the case of non-radial weights remains open. Recently, Peláez, Rätty\"a, and Wick \cite{PRW} proved that for a class of radially doubling weights, the reproducing kernel also admits a Shimorin-type integral representation.

Recall that a weight $\omega$ is a positive integrable function on the unit disk $\mathbb{D}$. It is said to be radial if $\omega(z)=\omega(|z|)$ for any $z\in\mathbb{D}$.
A weight $\omega$ (not necessarily radial) is called logarithmically subharmonic and reproducing at the origin if it satisfies
\begin{itemize}
    \item $\omega$ is logarithmically subharmonic, i.e., $\log \omega$ is a subharmonic function on $\mathbb{D}$; 
    \item $\omega$ is reproducing at the origin: for all polynomials $f\in\mathbb{C}[z]$, 
    $$
    f(0) = \int_{\mathbb{D}} f(z) \omega(z) \, dA(z).
    $$
\end{itemize}
Typical examples include $\omega(z) = |G(z)|^p$, where $G$ is an inner function in the Bergman space, and the standard weights
$
\omega_\alpha(z) = (\alpha+1)(1 - |z|^2)^{\alpha}
$
for $-1 < \alpha \leq 0$. Logarithmically subharmonic weights have a natural differential geometric interpretation. For such a weight $\omega$, if we equip the unit disk with the Riemannian metric $\sqrt{\omega(z)}|dz|$, the resulting Riemannian manifold is hyperbolic (i.e., its Gaussian curvature is negative everywhere). This weight plays a role in the study of the squared Laplace-Beltrami operator $\Delta \omega^{-1} \Delta$ on the resulting Riemannian manifold (see, e.g., \cite{HJS02} and \cite[Chapter 9]{HKZ}). Further results and applications related to logarithmically subharmonic weights can be found in \cite{Shi01} for the wandering subspaces problem; in \cite{HS02} for Hele--Shaw flow on hyperbolic surfaces; and in \cite{Hed00} for off-diagonal estimates of Bergman kernels.

In this paper, we revisit $L^p$--$L^q$ estimates for the Shimorin-type integral operator $T_\nu$ with an arbitrary finite positive Borel measure $\nu$ on $[0,1]$. We also note that the integral kernels are not, in general, the reproducing kernels of weighted Bergman spaces. The main question addressed in this paper is:
\begin{question}\label{pq-bound}
Let \( \nu \) be a finite positive Borel measure on \( [0,1] \). Characterize all pairs \((p,q) \in [1,\infty] \times [1,\infty]\) such that the Shimorin-type integral operator
\[
  T_\nu : L^p(\mathbb D) \longrightarrow L^q(\mathbb D)
\text{ is bounded.}
\]
\end{question}

Our motivation for investigating Question \ref{pq-bound} stems from the recent work of Cheng, Fang, Wang and Yu \cite{CFWY}, where the authors provide complete characterizations of the $L^p$-$L^q$ boundedness of Bergman-type operators for all real numbers $\alpha>0$. By comparing it with the standard Bergman projection, they classify $K_\alpha$ as hyper-singular for $2<\alpha < 3$ and sub-singular for $0<\alpha\leq 2$. Their results can be summarized by the $(1/p,1/q)$-graphs in Figure \ref{alpha2} and Figure \ref{alpha3}, where the red line separates bounded and unbounded $K_\alpha$. In the sub-singular case, the red line in Figure \ref{alpha2} is given by
\begin{align*}
\left\{ 
  \left( \frac{1}{p}, \frac{1}{q} \right) 
  :\ 1 \leq p \leq \frac{2}{2-\alpha},\ \frac{1}{q} = \frac{1}{p} + \frac{\alpha}{2} - 1 
\right\}.
\end{align*}
The points in the interior of the red line correspond to bounded $K_\alpha$, and they are unbounded at endpoints. In the hyper-singular case, every point along the red line yields an unbounded $K_\alpha$.
\vspace{-0.7cm}
\begin{center}
\begin{figure}[htbp]
    \centering
    %\vspace{-0.7cm}
    % 统一设置全局样式，确保两图视觉参数一致
      \tikzset{
        dot/.style={circle, fill, inner sep=0.25pt}, % 实心点（保持原有大小）
        hollow dot/.style={
            circle, 
            draw=red,        % 红色边框（高对比度）
            thick,           % 加粗边框（比默认粗2倍）
            inner sep=0.4pt  % 增大空心圆尺寸（原0.25pt→0.4pt）
        },
        base line/.style={gray!70, line width=0.6pt},
        red line/.style={red, line width=0.8pt},
        fill area/.style={gray!20}
    }
    
    % 左侧 Figure 1（α=1.5，0<α≤2）
    \begin{minipage}[t]{0.48\textwidth}
        \centering
        % 强制统一图形尺寸：固定宽高，确保与右侧一致
        \begin{tikzpicture}[scale=3, baseline={(0,0)}]
            % 1. 绘制单位矩形（统一边界）
            
            \fill[fill area] (0,0) -- (0,1) -- (1,1) -- (1,0.75) -- (0.25,0) -- cycle;
            \draw[base line] (0,0) rectangle (1,1);

            % 2. 标记矩形顶点（位置/样式完全统一）
            %\fill[dot] (0,1) node[left] {$(0,1)$};
            \fill[dot] (1,1) node[right] {$(1,1)$};
            %\fill[dot] (1,0) node[right] {$(1,0)$};
            \fill[dot] (0,0) node[left] {$(0,0)$};
            
            % 3. 自定义红色点（位置精准，样式统一）
            \node[circle, draw=red, thick, inner sep=1pt] at (1,0.75) {}; 
            \node[right, xshift=0cm, yshift=0.0cm] at (1,0.75) {$(1,\alpha/2)$};
            %\fill[dot, red] (1,0.75) node[right] {$(1,1.5/2)$};
            %\fill[dot, red] (0.25,0) node[below] {$((2-1.5)/2,0)$};
            %\fill[dot, red] (0.25,0) node[below, xshift=0.3cm, yshift=0.9cm] {$((2-\alpha)/2,0)$};
          \node[circle, draw=red, thick, inner sep=1pt] at (0.25,0) {}; 
          \node[below, xshift=0.3cm, yshift=0.0cm] at (0.25,0) {$(1-\alpha/2,0)$}; % 标注文字
            
            % 4. 红色连线
            \draw[red line] (0.25,0) -- (1,0.75);
            
            % 5. 填充灰色区域
            %\fill[fill area] (0,0) -- (0,1) -- (1,1) -- (1,0.75) -- (0.25,0) -- cycle;
            
            % 强制统一边界框：覆盖所有元素，确保尺寸与右侧完全一致
            \useasboundingbox (0,-0.2) rectangle (1.8,1.2);
        \end{tikzpicture}
        \caption{$K_\alpha\ (0 < \alpha \leq 2)$}
        \label{alpha2}
    \end{minipage}
    \hspace{-0.1\textwidth} % 极小间距，保持紧凑
    % 右侧 Figure 2（α=2.5，2<α<3）
    \begin{minipage}[t]{0.48\textwidth}
        %\centering
        % 完全匹配左侧的scale和baseline，确保尺寸/对齐一致
        \begin{tikzpicture}[scale=3, baseline={(0,0)}]
            % 1. 绘制单位矩形（与左侧完全相同）
            
            \fill[fill area] (0,0.5) -- (0,1) -- (0.5,1) -- cycle;
            \draw[base line] (0,0) rectangle (1,1);
            
            % 2. 标记矩形顶点（位置/样式与左侧一一对应）
            %\fill[dot] (0,1) node[left] {$(0,1)$};
            \fill[dot] (1,1) node[right] {$(1,1)$};
            %\fill[dot] (1,0) node[right] {$(1,0)$};
            \fill[dot] (0,0) node[left] {$(0,0)$};
            
            % 3. 自定义红色点（位置精准，样式统一）
            %\fill[dot, red] (0,0.5) node[left] {$(0,2.5-2)$};
                   \node[circle, draw=red, thick, inner sep=1pt] at (0,0.5) {}; 
          \node[left, xshift=0cm, yshift=0.0cm] at (0, 0.5) {$(0, \alpha-2)$};
            
            %\fill[dot, red] (0.5,1) node[above] {$(3-2.5,1)$};
              \node[circle, draw=red, thick, inner sep=1pt] at (0.5,1) {}; 
          \node[above, xshift=0cm, yshift=0.0cm] at (0.5,1) {$(3-\alpha,1)$};
          
            % 4. 红色虚线连线（仅线型不同，其他参数与左侧红线一致）
            \draw[red line, dashed] (0,0.5) -- (0.5,1);
            
            % 5. 填充灰色区域
            
            % 强制统一边界框：与左侧完全相同，确保整体尺寸一致
            \useasboundingbox (0,-0.2) rectangle (1.8,1.2);
        \end{tikzpicture}
        \vspace{2.5pt}
        \caption{$K_\alpha\ (2 < \alpha < 3)$}
        \label{alpha3}
    \end{minipage}
\end{figure}
\end{center}
\vspace{-1cm}
A more general form of the Bergman-type integral operator is known in the literature as Forelli-Rudin type operators, which have a long history and have been extensively studied; see \cite{ForelliRudin1974}, \cite{Liu2015}, \cite{Zhao2015}, \cite{Kaptanoglu2019}, \cite{ZhaoZhou2022}, \cite{Zhu2022} for more details.

Figures \ref{alpha2} and \ref{alpha3} reveal that the parameter $\alpha$ plays a crucial role in identifying the boundary of the $(1/p,1/q)$-graph such that the Bergman-type operator $K_\alpha:L^p \to L^q$ is bounded. Such a parameter does not explicitly appear in the Shimorin-type operator. To resolve this issue, we introduce a quantity of the measure $\nu$, denoted by $c_\nu$, as follows:
\begin{align}\label{critical-index}
c_\nu = \sup\left\{1 \le c < 2: \int_{0}^{1} \frac{1}{(1 - r^2)^{ 2/c'}} \,d\nu(r) < \infty\right\},
\end{align}
where $c'$  denotes the dual index of $c$, namely $1/c+1/c'=1$. We note that $c_{\delta_1}=1$, $c_{\delta_0}=2$ and $c_\nu=2$ if $\nu$ is the Lebesgue measure. Moreover, if $\nu(\{1\})\not=0$, then $c_\nu=1$.

For any finite positive Borel measure $\nu$ on $[0,1]$, we have a universal estimate
\[
\left|\frac{1}{1 - z\overline{\lambda}}\int_0^1 \frac{1}{1 - rz\overline{\lambda}} \,d\nu(r)\right|\leq\frac{2\nu([0,1])}{|1-z\overline{\lambda}|^2}.
\]
Hence, the Shimorin-type operator is sub-singular in the sense of \cite{CFWY}. On the other hand, by \cite[Theorem 1]{S}, for every $\alpha\in (1,2)$, 
\[
\frac{1}{(1-z\overline{\lambda})^\alpha}=\frac{1}{1 - z\overline{\lambda}}\int_0^1 \frac{1}{1 - rz\overline{\lambda}} \,d\nu_\alpha(r),
\]
where
\begin{align}\label{nualpha}
\nu_\alpha=\frac{r^{\alpha-2}(1 - r)^{-\alpha+1}}{\Gamma(\alpha - 1)\Gamma(-\alpha + 2)}dr.
\end{align}
%We shall include details of determining the measure \eqref{nualpha} in Subsection \ref{section:nualpha}. 
This suggests that the Shimorin-type operator $T_\nu$ exhibits $L^p$-$L^q$ estimates in a similar manner to the sub-singular Bergman-type operator. By \eqref{nualpha}, for $\alpha\in(1,2)$, we observe that 
$
c_{\nu_\alpha}=2/\alpha. 
$
Then the red line in Figure \ref{alpha2} can be written as
\[
\left\{\left(\frac{1}{p},\frac{1}{q}\right): 
1 \leq p \leq c_{\nu_\alpha}', \quad 
\frac{1}{q} = \frac{1}{p} + \frac{1}{c_{\nu_\alpha}} - 1 
\right\}.
\]
Combining this with the sub-singularity of Shimorin-type kernels, we are naturally led to the following conjectures:

\begin{conjecture}
The boundary line $\mathcal{C}$ between the $L^p$-$L^q$ boundedness and unboundedness regions of $T_\nu$ is given by
    \begin{align*}
     \mathcal{C} = \left\{ \left( \frac{1}{p},\frac{1}{q} \right) : 1 \leq p \leq c_\nu', \ \frac{1}{q} = \frac{1}{p} + \frac{1}{c_\nu} - 1 \right\}.
    \end{align*}
\end{conjecture}

\begin{conjecture}\label{pqboundary}
The operator $T_\nu:L^p(\mathbb{D})\to L^q(\mathbb{D})$ is bounded on the interior of the boundary line $\mathcal{C}$, while it is not bounded at the endpoints of the boundary line. 
\end{conjecture}
This paper resolves both conjectures. The first one is Theorem \ref{pq-necessary}. 
%\smallskip
\begin{theorem}\label{pq-necessary}
Let $\nu$ be a finite positive Borel measure on $[0,1]$ and $(p,q)\in [1,\infty]^2\setminus \mathcal{C}$.
Then $T_{\nu} : L^p(\D) \to L^q(\D)$ is bounded if and only if $(p,q)$ satisfies one of the following conditions:
    \begin{itemize}
    \item[(a)] $p=1$, $1\leq q < c_{\nu}$;
    \item[(b)] $1 < p < c_{\nu}'$, and
    \[
    \frac{1}{q} > \frac{1}{p} + \frac{1}{c_{\nu}} - 1;
    \]
    \item [(c)] $p = c_{\nu}'$, $1\leq q < \infty$;
    \item [(d)] $p > c_{\nu}'$, $1 \le q \le \infty$.
  \end{itemize}
\end{theorem}

The critical boundary line for the $(1/p, 1/q)$-graph of the operator $T_\nu$ is determined by $c_\nu$ and is 
illustrated in Figure \ref{theorem 1 graph}.

%\vspace{-0.6cm}

\begin{center}
\begin{figure}[htbp]
\centering  
%\vspace{-1cm} 
\begin{tikzpicture}[scale=3.5, baseline={(1.5,1.5)},
  % 样式定义
  base line/.style={black, thin},    
  fill area/.style={gray!20},        
  dot/.style={circle, fill=black, inner sep=1pt}, 
  blue line/.style={blue, thick}, 
  red line/.style={red, thick},
  % 端点样式
  endpoint blue/.style={circle, fill=blue, inner sep=1.3pt},
  endpoint red/.style={circle, fill=red, inner sep=1.3pt}
]
  % 1. 填充阴影区域 (蓝色中线左上方的区域)
  % 定义蓝色线的参数 1/c_nu (例如取 0.7)
  \def\invc{0.7} 
  \def\invcprime{0.3} % 1 - 1/c_nu
  
  %\fill[fill area] (0,0) -- (0,1) -- (1,1) -- (1,\invc) -- (\invcprime,0) -- cycle;

  % 2. 绘制单位矩形边框
  \draw[base line] (0,0) rectangle (1,1);
  
  % 3. 绘制三条线段并标注端点
  
  % --- (A) c_nu = 1: 红色 (对角线) ---
  \draw[red line] (0,0) -- (1,1);
  \node[endpoint red] at (0,0) {};
  \node[endpoint red] at (1,1) {};
  \node[left=2pt] at (0,0) {$(0,0)$};
  \node[right=2pt] at (1,1) {$(1,1)$};
  \node[red, font=\small, rotate=45] at (0.45, 0.62) {$c_\nu=1$};

  % --- (B) 中间线: 蓝色 ---
  \draw[blue line] (\invcprime,0) -- (1,\invc);
  \node[endpoint blue] at (\invcprime,0) {};
  \node[endpoint blue] at (1,\invc) {};
  % 端点坐标标注
  \node[below=3pt] at (\invcprime - 0.05,0) {$(1/c_\nu', 0)$};
  \node[right=2pt] at (1,\invc) {$(1, 1/c_\nu)$};
  % 方程标注
  \node[blue, font=\small, rotate=45] at (0.6, 0.42) {$\frac{1}{q}=\frac{1}{p}+\frac{1}{c_\nu}-1$};
  
  % --- (C) c_nu = 2: 红色 ---
  \draw[red line] (0.5,0) -- (1,0.5);
  \node[endpoint red] at (0.5,0) {};
  \node[endpoint red] at (1,0.5) {};
  \node[below=3pt] at (0.7,0) {$(1/2, 0)$};
  \node[right=2pt] at (1,0.5) {$(1, 1/2)$};
  \node[red, font=\small, rotate=45] at (0.8, 0.2) {$c_\nu=2$};

  % 4. 坐标轴标签
  \node[below=12pt, font=\small] at (0.5,-0.1) {$1/p$};
  \node[left=12pt, font=\small] at (0,0.5) {$1/q$};
  
  % 5. 辅助标记 (0,1) 和 (1,0) 使图像完整
 % \fill[dot] (0,1) node[left=2pt] {$(0,1)$};
 % \fill[dot] (1,0) node[right=2pt] {$(1,0)$};
  
  % 6. 强制边界框
  \useasboundingbox (-0.3,-0.3) rectangle (1.4,1.2);
\end{tikzpicture}
\vspace{-0.5cm}
\caption{Critical boundary lines}
\label{theorem 1 graph}
\end{figure}
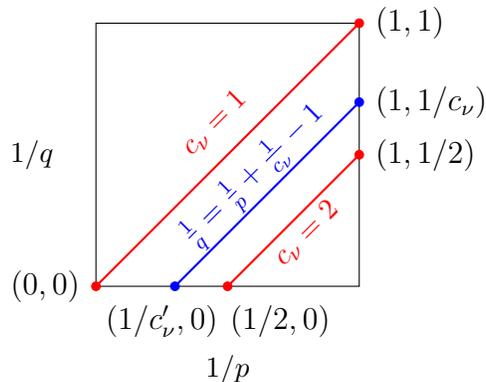
\end{center}

%\vspace{-0.5cm}

\begin{remark}
According to Theorem \ref{pq-necessary}, the boundary is defined by the line segment $$\frac{1}{q} = \frac{1}{p} + \frac{1}{c_\nu} - 1.$$ For the range $c_\nu \in [1, 2]$, the blue critical line always lies between the two red boundary lines:
\begin{itemize}
    \item The \textbf{upper red line} corresponds to $c_\nu = 1$, where the boundary coincides with the diagonal $1/q = 1/p$.
    \item The \textbf{lower red line} corresponds to $c_\nu = 2$, where the boundary is given by $1/q = 1/p - 1/2$.
\end{itemize}
Geometrically, as $c_\nu$ increases from $1$ to $2$, the blue line shifts downwards, indicating that the region of boundedness (the $(1/p, 1/q)$-graph) expands.
\end{remark}

Regarding Conjecture \ref{pqboundary}, consider the Lebesgue measure $\nu = dr$ on $[0,1]$. One can show that $T_\nu: L^p(\mathbb{D}) \to L^q(\mathbb{D})$ is not bounded for $1 < p < c_\nu'$ with the condition $$\frac{1}{q} = \frac{1}{p} + \frac{1}{c_\nu} - 1.$$ In addition, there exists a measure $\nu$ such that $T_\nu:L^1(\mathbb{D})\to L^{c_\nu}(\mathbb{D})$ is bounded. These examples are sharp, in contrast to the behavior of the sub-singular Bergman-type operator $K_\alpha$ for $0 < \alpha\leq 2$. Hence, Conjecture \ref{pqboundary} in general is not true, which naturally suggests the following problem. 
\begin{problem}\label{question-pqboundary}
What are the $L^p$-$L^q$ estimates for $(p,q)\in[1,\infty]\times[1,\infty]$ such that the index $(1/p,1/q)$ is in the boundary line $\mathcal{C}$ $($the blue line in Figure \ref{theorem 1 graph}$)$. 
\end{problem}
We give satisfactory answers to Problem \ref{question-pqboundary} which are partitioned into two cases:
\begin{itemize}
\item interior points of the blue line in Figure \ref{theorem 1 graph} and 
\item its corresponding endpoints. 
\end{itemize}
For the former case, we establish the $L^p$-$L^q$ estimate. For the latter case, the endpoints $(1, c_\nu)$ and $(c_\nu', \infty)$ form a dual pair. It is worth noting that the Bergman-type operator $K_\alpha$ fails to be bounded at these endpoints. By the theory of Riesz potential operators, $K_\alpha$ is of weak-type $(1, c_\nu)$. For the other endpoint $(c_\nu', \infty)$, the Fefferman-Stein $H^1$-BMO theory implies that the natural substitute for the target space $L^\infty$ is the BMO space. Along this line and a natural modified version of Problem \ref{question-pqboundary}, we use the following definition.

\begin{definition}
Let $\nu$ be a finite positive Borel measure on $[0,1]$, and let $c_\nu'$ denote the dual exponent of $c_\nu$. The Shimorin-type operator has a standard Bergman-type $L^p$–$L^q$ estimate if the following three conditions hold:
\begin{itemize}
    \item $T_{\nu}: L^1(\mathbb{D}) \to L^{c_{\nu},\infty}(\mathbb{D})$ is bounded.
    \item $T_\nu : L^p(\mathbb{D}) \to L^q(\mathbb{D})$ is bounded for $1 < p < c_{\nu}'$ with $q$ satisfying the relation
    \[
    \frac{1}{q} = \frac{1}{p} + \frac{1}{c_{\nu}} - 1.
    \]
    \item $T_{\nu}: L^{c_{\nu}'}(\mathbb{D}) \to \operatorname{BMO}(\mathbb{D})$ is bounded.
\end{itemize}
\end{definition}

\begin{theorem}\label{thm:new2}
Let $\nu$ be a finite positive Borel measure on $[0,1]$, and let $c_\nu'$ denote the dual exponent of $c_\nu$. 
The Shimorin-type operator has a standard Bergman-type $L^p$–$L^q$ estimate if and only if one of the following holds, according to the value of $c_\nu$:
\begin{itemize}
  \item When $c_\nu=1$, $\nu$ is a finite measure.
  \item When $1<c_\nu<2$, $\nu$ satisfies the $2/c_{\nu}'$ Carleson-type condition
  \begin{align}\label{carleson-type-0}
  \sup_{0 < t \le 1}\frac{\nu([1-t,1))}{t^{2/c_\nu'}} < \infty.
  \end{align}
  \item When $c_\nu=2$, $\nu$ is integrable with respect to the hyperbolic arc length on $[0,1)$, i.e.,
  \begin{align}\label{hyper-type-0}
  \int_{[0,1)} \frac{1}{1-r^2} \, d\nu(r) < \infty.
  \end{align}
\end{itemize}
\end{theorem}
\begin{remark}
Using dyadic harmonic analysis techniques, the most recent work by Hu and Zhou \cite{HZ25} proves that the hyper-singular Bergman-type operator \( K_\alpha \) (\( 2 < \alpha < 3 \)) is always weak-type \((p, q)\) bounded in the interior of the dashed line in Figure \ref{alpha3}.
\end{remark}
\begin{remark}
The proof of Theorem \ref{thm:new2} follows from Theorems \ref{weak bd on the line}, \ref{weak bd}, and \ref{dual of weak 1 alpha}, which in fact establish further results. Specifically:
\begin{itemize}
\item A single $L^p$–$L^q$ estimate, or weak-type/BMO-type estimate at the endpoints, is sufficient to imply the “only if” part of Theorem \ref{thm:new2}.
\item Theorem \ref{weak bd on the line} shows that strong-type $(p,q)$ estimates and weak-type $(p,q)$ estimates are equivalent in the interior of the boundary line.
\end{itemize}
\end{remark}
\begin{remark}
  The integral expression \eqref{hyper-type-0} has appeared in Theorem 1 of \cite{PRW}.
\end{remark}

\noindent \textbf{Applications to the Shimorin problem.} 
Let us return to Shimorin's representation problem \cite[Problem~19]{AHR}. Shimorin provided a solution for the radial case, and this led to the development of Shimorin-type integrals. Pel\'{a}ez, R\"{a}tty\"{a}, and Wick showed in \cite{PRW} that Shimorin-type integrals can also represent other classes of kernels. A natural question then arises: Given a kernel $K(z,\lambda)$, when does it admit a Shimorin-type integral representation?

Observe that for any finite positive Borel measure $\nu$ on $[0,1]$, we have $c_\nu \in [1,2]$. By Theorem \ref{pq-necessary}, if for a kernel $K$ the associated integral operator
\[
T_K f(z) = \int_{\mathbb{D}} K(z,\lambda) f(\lambda) \, dA(\lambda)
\]
fails to satisfy the $(1/p,1/q)$ condition described therein for every admissible pair $(p,q)$, then $K$ does not admit a Shimorin-type integral representation. In particular, we conclude that for $\alpha \in (0,1)$, the kernel  
\[
\frac{1}{(1 - z\overline{\lambda})^{\alpha}}
\]
does not have a Shimorin-type integral representation. Otherwise, by Theorem \ref{pq-necessary}, the Bergman-type operator $K_\alpha:L^p(\mathbb{D})\to L^p(\mathbb{D})$ is unbounded for all $1<p<2$ satisfying  
\[
\frac{1}{q} < \frac{1}{p} -\frac{1}{2},
\]
which contradicts the result in \cite{CFWY} (see, e.g., Figure \ref{alpha2}).

\bigskip

We conclude this introduction with an outline of the paper.

\begin{itemize}
    \item The first and most crucial step in establishing our main results is to determine the precise boundary between bounded and unbounded Shimorin-type operators. Our newly introduced quantity \(c_\nu\) identifies this boundary accurately.
    \item Using techniques from Bergman space theory, we derive effective $L^p$-norm estimates for the Shimorin-type kernels, as stated in Proposition \ref{pnorm of K} in Section \ref{section:two-side}.
    \item Combining these kernel estimates with the general theory of integral operators yields a complete proof of the sufficiency of Theorem \ref{pq-necessary} (see Subsection \ref{S:suf1}).
    \item The necessity direction of Theorem \ref{pq-necessary} is considerably more challenging than the sufficiency part. Its proof relies on a natural coefficient multiplier representation of \(T_\nu\) (see equation \eqref{tcm}).
    \item In Section \ref{S:mct}, we establish a sharp growth estimate for the coefficient multiplier sequence via Proposition \ref{prop:mn_limsup_estimate}. This estimate enables, among other applications, the construction of suitable test functions \eqref{testingft} and \eqref{testinggt}, which are used to derive the necessary condition in Subsection \ref{S:necethm1}. We emphasize that the quantity \(c_\nu\) plays a central role throughout these arguments.
    \item Under the Carleson condition \eqref{carleson-type-0} or the integrability condition \eqref{hyper-type-0}, we obtain Calderón--Zygmund-type estimates (size and smoothness estimates) for the Shimorin-type kernel in Lemmas \ref{u-estimate of K} and \ref{u-estimate of K-2}. These estimates, together with a fractional Bergman CZO argument, yield the ``if'' part of Theorem \ref{thm:new2} (presented at the beginning of Section \ref{Thm:2-3}).
    \item Finally, by constructing specific test functions, we prove Theorems \ref{weak bd on the line}, \ref{weak bd}, and \ref{dual of weak 1 alpha} in Section \ref{Thm:2-3}. This completes the proof of Theorem \ref{thm:new2}.
\end{itemize}

\medskip
{\flushleft\bf Acknowledgements.} Y.\,L and Z.\,W are supported by NSF of China (No.12471116) and 
the Fundamental Research Funds for the Central Universities (No.2025CDJ-IAIS YB-004). K.\,Z is supported by NSF of China (No.12231005) and  National Key R\&D Program of China (2024YFA1013400).

\section{Two-sided $L^p$-norm estimates for Shimorin-type kernels}\label{section:two-side}
We begin by recalling the Shimorin-type kernel defined in \eqref{sio-kernel}:
\[
K_{\nu}(z,\lambda) = \frac{1}{1 - z\overline{\lambda}} \int_0^1 \frac{1}{1 - rz\overline{\lambda}} \,d\nu(r).
\]
For a general finite positive Borel measure $\nu$ on $[0,1]$, we stress that this kernel need not coincide with the reproducing kernels of weighted Bergman spaces. Consequently, existing $L^p$-norm estimates for weighted Bergman kernels are insufficient for our analysis. The main objective of this section is to establish sharp two-sided $L^p$-norm bounds for $K_\nu$.
\begin{proposition}\label{pnorm of K}
Let $\nu$ be a finite positive Borel measure on $[0,1]$ such that $\nu(\{1\})=0$. For all $1 < p < \infty$ and each $z\in\mathbb{D}$, we have
\begin{align*}
(1-|z|^2)^{2/p-1} \int_{0}^1 \frac{d\nu(r)}{1 - r|z|^2} 
\le \|K_{\nu}(z,\cdot)\|_{L^p(\mathbb{D})} 
\le \int_0^1 \frac{1}{1 - r} \int_{r}^1 (1 - t|z|^2)^{2/p - 2} dt \, d\nu(r).
\end{align*}
\end{proposition}

We next give a simple double integral representation of the Shimorin-type kernel.
\begin{proposition}\label{kernel-doub}
  Let $\nu$ be a finite positive Borel measure on $[0,1]$ such that $\nu(\{1\})=0$. Then
  \[
    K_{\nu}(z,\lambda) = \int_0^1 \frac{1}{1 - r} \int_r^1 \frac{dt}{(1 - tz\overline{\lambda})^2} d\nu(r).
  \]
\end{proposition}
\begin{proof}
Let $\nu$ be a finite positive Borel measure on $[0,1]$ and $\nu(\{1\})=0$. Fix $z,\lambda\in \mathbb{D}$ with $z\overline{\lambda} \ne 0$ and $r\in[0,1)$. A direct computation gives the identity below. The case $z\overline{\lambda} = 0$ follows by continuity: 
\begin{align*}
\frac{1}{1-r}\int_{r}^{1} \frac{dt}{(1 - tz\overline{\lambda})^2} &= \frac{1}{1-r} \frac{1}{z\overline{\lambda}} \left (\frac{1}{1 - z\overline{\lambda}} - \frac{1}{1 - rz\overline{\lambda}}\right ) = \frac{1}{(1 - z\overline{\lambda})(1 - rz\overline{\lambda})}.
\end{align*}
It follows that
\[
K_{\nu}(z,\lambda) = \int_{0}^1 \frac{1}{(1 - z\overline{\lambda})(1 - rz\overline{\lambda})} d\nu(r) = \int_0^1 \frac{1}{1 - r} \int_r^1 \frac{dt}{(1 - tz\overline{\lambda})^2} d\nu(r).
\]
This completes the proof.
\end{proof}
By Proposition \ref{kernel-doub}, the operator \(T_{\nu}\) has a simple form acting on analytic functions.

\begin{corollary}\label{T_{nu}formula}
Let $\nu$ be a finite positive Borel measure on $[0,1]$ such that $\nu(\{1\})=0$. For \(f \in L_a^1(\mathbb{D})\), the following identity holds:
\[
T_{\nu}f(z) = \int_{0}^{1} \frac{1}{1 - r} \int_r^1 f(tz)\,dt \,d\nu(r), \quad z \in \mathbb{D}.
\]
\end{corollary}
\begin{proof}
Let \(f \in L_a^1(\mathbb{D})\). By Proposition \ref{kernel-doub},
\[
K_{\nu}(z,\lambda) = \int_0^1 \frac{1}{1 - r} \int_r^1 \frac{dt}{(1 - tz\overline{\lambda})^2} d\nu(r).
\]
By Fubini's theorem, we have
\begin{align*}
T_{\nu}f(z) &= \int_{\mathbb{D}} K_{\nu}(z,\lambda) f(\lambda) \,dA(\lambda) \\
&= \int_0^1 \frac{1}{1 - r} \int_r^1 \int_{\mathbb{D}} \frac{f(\lambda)}{(1 - tz\overline{\lambda})^2} \,dA(\lambda) dt \,d\nu(r) \\
&= \int_{0}^{1} \frac{1}{1 - r} \int_r^1 f(tz) \,dt \,d\nu(r),
\end{align*}
where the last equality is due to the Bergman reproducing property of analytic functions on \(\mathbb{D}\). This completes the proof.
\end{proof}

\begin{proof}[Proof of Proposition \ref{pnorm of K}]
For $1 \le p < \infty$, recall that for any $f\in L_a^p(\mathbb{D})$ and $z\in\mathbb{D}$, the estimate (see \cite{Zhu14}, Theorem 4.14)
\begin{align}\label{sub}
|f(z)| \leq \frac{\|f\|_{L^p(\mathbb{D})}}{\left(1 - |z|^2\right)^{2/p}} 
\end{align}
holds. Applying \eqref{sub} to the function $\overline{K_{\nu}(z,\cdot)} \in L_a^p(\mathbb{D})$, we deduce 
\[
\left|K_{\nu}(z,z)\right| \le \frac{\left\|K_{\nu}(z,\cdot)\right\|_{L^p(\mathbb{D})}}{(1 - \left|z\right|^2)^{2/p}} .
\]
Rearranging this inequality yields the lower bound:
\begin{align*}
\left\|K_{\nu}(z,\cdot)\right\|_{L^p(\mathbb{D})} &\ge (1 - \left|z\right|^2)^{2/p} \left|K_{\nu}(z,z)\right|\\
& = (1 - \left|z\right|^2)^{2/p - 1} \int_0^1 \frac{d\nu(r)}{1 - r\left|z\right|^2}.
\end{align*}

To complement this, we establish an upper bound for $\left\|K_{\nu}(z,\cdot)\right\|_{L^p(\mathbb{D})}$. First, note that the $L^p$-norm is invariant under complex conjugation, so
\[
\left\|K_{\nu}(z,\cdot)\right\|_{L^p(\mathbb{D})} = \left\|\overline{K_{\nu}(z,\cdot)}\right\|_{L^p(\mathbb{D})}.
\]
By the duality of $L^p$ and $L^{p'}$ spaces (where $1/p + 1/p' = 1$), we express the norm as a supremum over the unit sphere of $L_a^{p'}(\mathbb{D})$:
\begin{align*}
\left\|\overline{K_{\nu}(z,\cdot)}\right\|_{L^p(\mathbb{D})}
&= \sup_{\substack{f \in L_a^{p'}(\mathbb{D}), \\ \left\|f\right\|_{L^{p'}(\mathbb{D})} = 1}} \left| \left\langle \overline{K_{\nu}(z,\cdot)}, f \right\rangle_{L^2(\mathbb{D})} \right| \\
&= \sup_{\substack{f \in L_a^{p'}(\mathbb{D}), \\ \left\|f\right\|_{L^{p'}(\mathbb{D})} = 1}} \left| \int_{\mathbb{D}} \overline{K_{\nu}(z,\lambda)} \overline{f(\lambda)} \, dA(\lambda) \right|\\
&= \sup_{\substack{f \in L_a^{p'}(\mathbb{D}), \\ \left\|f\right\|_{L^{p'}(\mathbb{D})} = 1}} |T_\nu f(z)|,
\end{align*}
where the last equality follows from the definition of the operator $T_\nu$. 

By Corollary \ref{T_{nu}formula}, the action of $T_\nu$ on $f$ at $z$ has an integral representation, so we substitute to obtain:
\[
\left\|\overline{K_{\nu}(z,\cdot)}\right\|_{L^p(\mathbb{D})}=
\sup_{\substack{f \in L_a^{p'}(\mathbb{D}), \\ \left\|f\right\|_{L^{p'}(\mathbb{D})} = 1}}
\left|\int_{0}^{1} \frac{1}{1 - r} \int_r^1 f(tz)\,dt \,d\nu(r)\right|.
\]
Now, take any $f \in L_a^{p'}(\mathbb{D})$ with $\left\|f\right\|_{L^{p'}(\mathbb{D})} = 1$. Applying estimate \eqref{sub} (with $p$ replaced by $p'$) to $f(tz)$ for $t\in(0,1)$ and $z\in\mathbb{D}$, we have
\[
|f(tz)|\leq \frac{1}{\left(1 - |tz|^2\right)^{2/p'}} \le \frac{1}{\left(1 - t|z|^2\right)^{2/p'}}.
\]
Using the identity $-2/p' = 2/p - 2$, we bound the supremum by:
\begin{align*}
\left\|K_{\nu}(z,\cdot)\right\|_{L^p(\mathbb{D})} \le \int_0^1 \frac{1}{1 - r} \int_r^1 (1 - t\left|z\right|^2)^{2/p - 2} dt \, d\nu(r).
\end{align*}
This completes the proof.
\end{proof}

\begin{corollary}\label{Kp-norm}
    Let $\nu$ be a finite positive Borel measure on $[0,1]$ such that $\nu(\{1\})=0$.
  Suppose that $1<p<2$ and that $\int_{0}^1 (1-r)^{2/p-2}\,d\nu(r)<\infty$.
  Then
  \[
    \frac{1}{2}\int_0^1 (1 - r)^{2/p - 2} d\nu(r) \le \sup_{z\in \mathbb{D}} \|K_{\nu}(z,\cdot)\|_{L^p(\mathbb{D})} \le \frac{2}{2 - p} \int_0^1 (1 - r)^{2/p - 2} d\nu(r).
  \]
\end{corollary}

\begin{proof}
By Proposition \ref{pnorm of K}, for all $z \in \mathbb{D}$, we have
\[
\|K_{\nu}(z,\cdot)\|_{L^p(\mathbb{D})} \le  \int_0^1 \frac{1}{1 - r} \int_{r}^1 (1 - t|z|^2)^{2/p - 2} dt \, d\nu(r).
\]
Recall that for all $t \in [r,1]$, the function $x \mapsto (1 - tx^2)^{2/p - 2}$ is non-decreasing on $[0,1)$ (since $2/p - 2 < 0$ for $p > 1$). This monotonicity implies
\begin{align*}
\sup_{z\in \mathbb{D}} \|K_{\nu}(z,\cdot)\|_{L^p(\mathbb{D})} 
&\le  \int_0^1 \frac{1}{1 - r} \int_{r}^1 (1 - t|z|^2)^{2/p - 2} dt \, d\nu(r)\\
&\le  \int_0^1 \frac{1}{1 - r} \int_{r}^1 (1 - t)^{2/p - 2} dt \, d\nu(r)\\
&\le \frac{2}{2 - p} \int_0^1 (1 - r)^{2/p - 2} d\nu(r),
\end{align*}
where the final inequality follows from  
$$
\int_r^1 (1 - t)^{2/p - 2} dt = \frac{(1 - r)^{2/p - 1}}{2/p - 1} = \frac{p(1 - r)^{2/p - 1}}{2 - p} \le \frac{2(1 - r)^{2/p - 1}}{2 - p}.
$$ 
Conversely, give $r\in[0,1)$ and pick $z_0\in\mathbb{D}$  with $|z_0|^2=r$, we have
\begin{align*}
\sup_{z\in\mathbb{D}} \frac{(1 - |z|^2)^{2/p - 1}}{1 - r|z|^2} & \ge 
\frac{(1 - |z_0|^2)^{2/p - 1}}{1 - r|z_0|^2}=\frac{(1 - r)^{2/p - 1}}{1 - r^2}\\
&=\frac{(1 - r)^{2/p - 1}}{(1 - r)(1 + r)} \geq  \frac{1}{2} (1 - r)^{2/p - 2}.
\end{align*}
for all $r \in [0,1)$. 
By Proposition \ref{pnorm of K} and the Lebesgue Dominated Convergence Theorem, we conclude
\[
\sup_{z\in \mathbb{D}} \|K_{\nu}(z,\cdot)\|_{L^p(\mathbb{D})} \ge \frac{1}{2} \int_0^1 (1 - r)^{2/p - 2} d\nu(r).
\]
This completes the proof.
\end{proof}

Finally, we deal with the pointwise estimates of Shimorin-type kernel when $\nu(\{1\}) \ne 0$. We begin with the following elementary result, whose proof follows from a direct computation.
\begin{lemma}\label{zw estimate}
For all $z, \lambda \in \mathbb{D}$ and all $r \in [0,1]$, we have
\[
\frac{|1 - z\overline{\lambda}|}{|1 - rz\overline{\lambda}|} \le 2.
\]
\end{lemma}

\begin{proposition}\label{K estimates with nu1 neq 0}
  Let $\nu$ be a finite positive Borel measure on $[0,1]$, and define $\nu_1$ on $[0,1]$ given by the following: for any measurable set $I\subset [0,1]$, $\nu_1(I)=\nu([0,1)\cap I)$.
Thus,
\[
\nu = \nu_1 + \nu(\{1\})\delta_1.
\]
Then,
\[
\frac{1}{2} \left(\frac{\nu(\{1\})}{|1 - z\bar{\lambda}|^2} + |K_{\nu_{1}}(z,\lambda)|\right) \le |K_{\nu}(z,\lambda)| \le \frac{\nu(\{1\})}{|1 - z\bar{\lambda}|^2} + |K_{\nu_{1}}(z,\lambda)|.
\]
In particular, if $\nu(\{1\}) \ne 0$, then by Lemma \ref{zw estimate},
\[
\frac{\nu(\{1\})}{2|1 - z\bar{\lambda}|^2} \le |K_{\nu}(z,\lambda)| \le \frac{2 \nu([0,1])}{|1 - z\bar{\lambda}|^2}.
\]
\end{proposition}

\begin{proof}
  Recall the definition of $K_{\nu}$, we have 
  \begin{align*}
    |K_{\nu}(z,\lambda)| 
    &= \left| \frac{1}{1 - z\bar{\lambda}} \int_{0}^{1} \frac{1}{1 - rz\bar{\lambda}} d\nu(r) \right| \\
    &= \left| \frac{1}{1 - z\bar{\lambda}} \int_{0}^{1} \frac{1}{1 - rz\bar{\lambda}} d\nu_1(r) + \frac{\nu(\{1\})}{(1 - z\bar{\lambda})^2} \right| \\
    &\le |K_{\nu_{1}}(z,\lambda)| + \frac{\nu(\{1\})}{|1 - z\bar{\lambda}|^2}.
  \end{align*}
On the other hand, note that
\begin{align*}
  |K_{\nu}(z,\lambda)| 
  &= \left| \frac{1}{1 - z\bar{\lambda}} \int_{0}^{1} \frac{1}{1 - rz\bar{\lambda}} d\nu(r) \right| \\
  &= \frac{1}{|1 - z\bar{\lambda}|^2} \left| \int_{0}^{1} \frac{1 - z\bar{\lambda}}{1 - rz\bar{\lambda}} d\nu_{1}(r) + \nu(\{1\})\right|.
\end{align*}
We claim that for all $z, \lambda \in \mathbb{D}$ and $0 \le r \le 1$,
\[
\operatorname{Re} \frac{1 - z\bar{\lambda}}{1 - rz\bar{\lambda}} \ge 0.
\]
A direct calculation shows that 
\[
  \frac{1 - z\bar{\lambda}}{1 - rz\bar{\lambda}} = \frac{(1 - z\bar{\lambda})(1 - r\bar{z}\lambda)}{|1 - rz\bar{\lambda}|^2} = \frac{1 - z\bar{\lambda} - r \bar{z}\lambda + r|z|^2|\lambda|^2}{|1 - rz\bar{\lambda}|^2}.
\]
Then we take the real part of the equation above, and using the inequality $\operatorname{Re} (z\bar{\lambda}) \le |z\lambda|$,
\[
\operatorname{Re} \frac{1 - z\bar{\lambda}}{1 - rz\bar{\lambda}} \ge \frac{1 - |z\bar{\lambda}|(1 + r) + r|z|^2|\lambda|^2}{|1 - rz\bar{\lambda}|^2} = \frac{(1 - |z| |\lambda|)(1 - r|z||\lambda|)}{|1 - rz\bar{\lambda}|^2} \ge 0.
\]
Since the real part of $\int_{0}^{1} \frac{1 - z\bar{\lambda}}{1 - rz\bar{\lambda}} d\nu_{1}(r)$ is positive, then
\[
\left| \int_{0}^{1} \frac{1 - z\bar{\lambda}}{1 - rz\bar{\lambda}} d\nu_{1}(r) + \nu(\{1\})\right| \ge \left| \int_{0}^{1} \frac{1 - z\bar{\lambda}}{1 - rz\bar{\lambda}} d\nu_{1}(r)\right|,
\]
and
\[
 \left| \int_{0}^{1} \frac{1 - z\bar{\lambda}}{1 - rz\bar{\lambda}} d\nu_{1}(r) + \nu(\{1\})\right| \ge \nu(\{1\}).
\]
Combining the inequalities above, we have 
\begin{align*}
|K_{\nu}(z,\lambda)| 
&\ge \frac{1}{2} \frac{1}{|1 - z\bar{\lambda}|^2} \left( \left| \int_{0}^{1} \frac{1 - z\bar{\lambda}}{1 - rz\bar{\lambda}} d\nu_{1}(r)\right| + \nu(\{1\})\right)\\
&= \frac{1}{2} \left(\frac{\nu(\{1\})}{|1 - z\bar{\lambda}|^2} + |K_{\nu_{1}}(z,\lambda)|\right).
\end{align*}
This completes the proof.
\end{proof}

\section{The multiplier coefficient related to $T_\nu$}\label{S:mct}
Using an argument analogous to that of Lemma 11 in \cite{CFWY}, we have
\[
T_\nu P f = T_\nu f
\]
for all $1 < p < \infty$ and every $f \in L^p(\mathbb{D})$, where $P$ denotes the Bergman projection. This identity implies that, for $p > 1$, the boundedness of the operator $T_\nu$ on $L^p(\mathbb{D})$ is equivalent to its boundedness on the Bergman space $L_a^p(\mathbb{D})$. As a consequence, the multiplier coefficient technique on analytic function spaces yields a powerful tool for $T_\nu$. 

For $1 \leq p < \infty$, a sequence $\{m_n\}_{n=0}^\infty$ of complex numbers is called a coefficient multiplier defined on $L_a^p(\mathbb{D})$ is the mapping
\[
M_m f: \sum_{n=0}^\infty a_n z^n \mapsto \sum_{n=0}^\infty m_n a_n z^n,
\]
where $f \in L_a^p(\mathbb{D})$ with Taylor expansion 
$f(z) = \sum_{n=0}^\infty a_n z^n.$
With this definition in hand, a direct computation yields the Taylor series expansion of $T_\nu f$:
\begin{align}\label{tcm}
  T_{\nu} f(z) = \sum_{n=0}^{\infty} \left( \frac{1}{n+1} \int_0^1 \frac{1 - r^{n+1}}{1 - r} d\nu(r) \right) a_n z^n.
\end{align}
Then the operator $T_\nu$ is identified with the coefficient multiplier sequence \[\mathfrak{m} = \{m_n\}_{n=0}^\infty\]
where
\begin{align}\label{momentsequence}
m_n = \frac{1}{n+1} \int_0^1 \frac{1 - r^{n+1}}{1 - r} d\nu(r), \quad n \geq 0.
\end{align}

The following elementary result illustrates the growth property of the coefficient multiplier sequence. Let
\begin{align}\label{eq:s0_def}
s_0 = \sup\left\{ 0 \le s < 1 \,:\, \int_0^1 \frac{1}{(1 - r)^s} \, d\nu(r) < \infty \right\}.
\end{align}
Recall the quantity $c_\nu$ defined by \eqref{critical-index},
\begin{align*}
c_\nu = \sup\left\{1 \le c < 2: \int_{0}^{1} \frac{1}{(1-r^2)^{2/c'}} \,d\nu(r) < \infty\right\}.
\end{align*}
Then combining the definition of $c_{\nu}$ and \eqref{eq:s0_def}, we have 
\[
s_0 = \frac{2}{c_{\nu}'}.
\]
\begin{proposition}\label{prop:mn_limsup_estimate}
Let $\nu$ be a finite positive Borel measure on $[0,1]$ such that $\nu(\{1\}) = 0$, and let $s_0$ be defined as in \eqref{eq:s0_def}. If $0 < s_0 < 1$, we have
\[
\limsup_{n \to \infty} \frac{\log m_n}{\log(n+1)} = -s_0.
\]
\end{proposition}
For the proof, we need the following lemma.
\begin{lemma}\label{mnestmate1}
  Let $\nu$ be a finite positive Borel measure on $[0,1]$ such that $\nu(\{1\}) = 0$.
  For $0 < s < 1$, 
  \begin{align}\label{Main-LR}
\int_{0}^{1} (1 - r)^{-s} d\nu(r) < \infty\quad \text{ if and only if } \quad \sum_{n=0}^{\infty} m_n (n+1)^{s - 1}  < \infty.
  \end{align}
\end{lemma}

\begin{proof}
Let $\varphi(t) = 1 - t$ be a function defined on $[0,1]$, and let $\mu$ denote the pushforward measure of $\nu$ under $\varphi$. Specifically, for any measurable set $E$,
\begin{equation}\label{def mu}
  \mu(E) = \nu(\varphi^{-1}(E)) = \nu\left(\left\{t \in [0,1] : 1 - t \in E\right\}\right).
\end{equation}
By the change-of-variable formula, for any $s \in (0,1)$, we have
\begin{equation}\label{left-00}
\int_0^1 (1 - r)^{-s} \,d\nu(r) = \int_0^1 t^{-s} \,d\mu(t).
\end{equation}

\begin{claim}\label{claim1}
For any non-negative integer $n$, we have
\begin{equation}\label{claim111}
\left(1 - e^{-1}\right)\int_{0}^1 \min\left\{ 1, \frac{1}{(n+1)t} \right\}d\mu(t) \le m_n \le  \int_{0}^1 \min\left\{ 1, \frac{1}{(n+1)t} \right\}d\mu(t).
\end{equation}
\end{claim}

Assuming the validity of Claim \ref{claim1}, the series
\begin{equation*}
\sum_{n=0}^{\infty}m_n (n+1)^{s-1} < \infty
\end{equation*}
if and only if 
\begin{equation}\label{sums}
\sum_{n=0}^{\infty} (n+1)^{s-1}\int_0^1 \min\left\{ 1, \frac{1}{(n+1)t} \right\}d\mu(t) < \infty.
\end{equation}

For a fixed $s \in (0,1)$, we define a function $S_s: (0,1] \to \mathbb{R}^+$ by
\[
S_s(t) = \sum_{n=0}^{\infty}(n+1)^{s - 1} \min\left\{1, \frac{1}{(n+1)t}\right\}.
\]
By the Monotone Convergence Theorem, the integral in \eqref{sums} can be rewritten as
\[
\int_{0}^{1} S_s(t)d\mu(t) = \sum_{n=0}^{\infty} (n+1)^{s-1}\int_0^1 \min\left\{ 1, \frac{1}{(n+1)t} \right\}d\mu(t).
\]
Thus, for $s \in (0,1)$,
\begin{equation}\label{R3}
\sum_{n=0}^{\infty}m_n (n+1)^{s-1} < \infty
\end{equation}
if and only if 
\begin{equation}\label{R4}
\int_{0}^{1} S_s(t)d\mu(t) < \infty.
\end{equation}

\begin{claim}\label{claimLR}
For each $s \in (0,1)$, there exist positive constants $C_1$ and $C_2$ such that for all $t \in (0,1]$,
\begin{equation}\label{LR}
C_1 t^{-s} \le S_s(t) \le C_2 t^{-s}.
\end{equation}
\end{claim}

Once Claims \ref{claim1} and \ref{claimLR} are established, the desired conclusion \eqref{Main-LR} follows immediately from \eqref{left-00}, \eqref{LR}, \eqref{R4}, and \eqref{R3}.

\medskip

We now proceed to verify Claim \ref{claimLR}.
\begin{proof}[Proof of Claim \ref{claimLR}]
For $t\in(0,1]$, let $N_t = \lfloor 1/t \rfloor$ denote the largest integer not exceeding $1/t$, i.e., $N_t \le 1/t < N_t + 1$. Then, we have
\begin{equation*}
    S_s(t)
    = \sum_{n=0}^{N_t - 1} (n+1)^{s - 1} + \frac{1}{t} \sum_{n=N_t}^{\infty} (n+1)^{s - 2}.
\end{equation*}
For $0 < t \leq 1$, the definition of $N_t = \lfloor 1/t \rfloor$ implies $N_t + 1 \geq t^{-1} \geq 1$.
Moreover, $N_t + 1 \geq 2$. We thus derive the lower bound:
\begin{equation}\label{lower_bound}
\begin{aligned}
  (N_t+1)^s - 1 &= (N_t+1)^s \left(1 - (N_t+1)^{-s}\right) \\
  &\geq (N_t+1)^s \left(1 - 2^{-s}\right) \\
  &\geq t^{-s} \left(1 - 2^{-s}\right).
\end{aligned}
\end{equation}
Recall the elementary integral test bound:
\[
    \int_{1}^{N_{t}+1} x^{s-1} dx \leq \sum_{n=0}^{N_t - 1} (n+1)^{s-1} \leq 1 + \int_{1}^{N_{t}} x^{s-1} d x.
\]
Substituting the integral result gives
\begin{equation}\label{Ss part 1}
	\frac{(N_t+1)^s - 1}{s}\leq\sum_{n=0}^{N_t - 1} (n+1)^{s-1}\leq 1+\frac{(N_t)^s - 1}{s}.
\end{equation}
Combining \eqref{lower_bound} and the left part of \eqref{Ss part 1}, we obtain
\begin{equation}\label{St 1 up}
  \sum_{n=0}^{N_t - 1} (n+1)^{s-1} \ge \frac{1 - 2^{-s}}{s}\,t^{-s}.
\end{equation}
For the complementary upper bound, we use the estimate $N_t = \lfloor 1/t \rfloor \le  t^{-1}$ for  $0 < t \leq 1$. Substituting the result in the right part of \eqref{Ss part 1} gives:
\[
  \sum_{n=0}^{N_t - 1} (n+1)^{s-1} \le 1 + \frac{(N_t)^s - 1}{s} \le 1 + \frac{(N_t)^s}{s} \le 1 + \frac{1}{s}\,t^{-s}.
\]
Since $0 < t \leq 1$ implies $t^{-s} \geq 1$, we absorb the additive constant $1$ into the leading term to get the compact upper bound:
\begin{equation}\label{St 1 lower}
  \sum_{n=0}^{N_t - 1} (n+1)^{s-1} \le \left(1 + \frac{1}{s}\right) t^{-s}.
\end{equation}
Combining \eqref{St 1 up} and \eqref{St 1 lower}, we establish the two-sided bound:
\begin{equation}\label{St1}
\frac{1 - 2^{-s}}{s}\,t^{-s} \leq \sum_{n=0}^{N_t - 1} (n+1)^{s-1} \leq \left(1 + \frac{1}{s}\right) t^{-s}.
\end{equation}

Next, we estimate
\[
\frac{1}{t}\sum_{n=N_t}^{\infty} (n+1)^{s-2}.
\]
We also use the elementary integral test bound
\[
\int_{N_t+1}^{\infty} x^{s-2} dx \le \sum_{n=N_t}^{\infty} (n+1)^{s-2} \le \int_{N_t}^{\infty} x^{s-2} d x.
\]
Substituting the integral result gives
\begin{equation}\label{Ss part 2}
 \frac{(N_t+1)^{s-1}}{1-s} \le \sum_{n=N_t}^{\infty} (n+1)^{s-2} \le \frac{(N_t)^{s-1}}{1-s}.
\end{equation}
Noting that for $0 < t \leq 1/2$, 
\[
N_t=\left\lfloor \frac{1}{t}\right\rfloor \ge \frac{1}{t}-1 \ge \frac{1}{2t},
\]
and for $1/2 < t \le 1$, 
\[
	N_t=\left\lfloor \frac{1}{t}\right\rfloor \ge 1 \ge \frac{1}{2t}.
\]
Therefore, in all cases we have \(N_t \ge 1/(2t)\). Moreover,
\[
N_t+1 \le \frac{1}{t}+1 \le \frac{2}{t}.
\]
Combining this with \eqref{Ss part 2}, we have
\begin{equation}\label{St 2}
\frac{2^{s - 1}}{1- s} t^{-s} \le \frac{1}{t}\sum_{n=N_t}^{\infty} (n+1)^{s-2} \le \frac{2^{1-s}}{1-s} t^{-s}.
\end{equation}

Finally, combining the estimates in \eqref{St1} and \eqref{St 2}, we conclude that there exist positive constants
\[
C_1 = \frac{1 - 2^{-s}}{s} + \frac{2^{s - 1}}{1 - s} \quad \text{and} \quad C_2 = 1 + \frac{1}{s} + \frac{2^{1-s}}{1-s}
\]
such that
\[
 C_1 t^{-s} \le S_s(t) \le C_2 t^{-s}
\]
for all $0<t\le 1$. This yields the desired inequality \eqref{LR} and completes the proof of Claim \ref{claimLR}.
\end{proof}

It remains to verify Claim \ref{claim1}.

\begin{proof}[Proof of Claim \ref{claim1}]
For any $n\geq 0$, set $N = n+1$. 
Using the change-of-variable formula with $t = 1 - r$, we have
  \[
  m_n = \int_0^1 \frac{1 - (1-t)^N}{Nt} d\mu(t),
  \]
  where $\mu$ is defined by \eqref{def mu}.
Hence, it suffices to show for any $t \in (0,1]$,
  \[
  \left(1 - e^{-1}\right) \min\left\{ 1, \frac{1}{Nt} \right\}\le \frac{1 - (1-t)^N}{Nt} \le  \min\left\{ 1, \frac{1}{Nt} \right\}.
  \]

\noindent\textbf{The case $0 < t \le 1/N$.} In this case,
  \begin{align}\label{small00}
  \min\left\{ 1, \frac{1}{Nt} \right\} = 1.
  \end{align}
  By Bernoulli's inequality $(1-t)^N \ge 1 - Nt$, we obtain:
  \begin{align}\label{smallupper}
\frac{1 - (1-t)^N}{Nt} \le \frac{Nt}{Nt} = 1.
  \end{align}
  For the lower bound, define
  \[
  h(t) = \frac{1 - (1-t)^N}{Nt},  \quad 0 < t \le \frac{1}{N}.
   \]
  If $N = 1$, then $h(t) = 1$. The inequality holds trivially for $N = 1$. Hence, we assume $N \ge 2$, and compute the derivative:
  \begin{align*}
    h'(t) &= \frac{N(1 - t)^{N-1} \cdot Nt - N\left(1 - (1 - t)^{N}\right)}{N^2t^2} \\
    &= \frac{Nt(1 - t)^{N-1} + (1 - t)^N - 1 }{Nt^2}.
  \end{align*}
  Let $\phi(t) = Nt(1 - t)^{N-1} + (1 - t)^N - 1$. Then
  \begin{align*}
    \phi'(t) 
    &= N(1 - t)^{N-1} - Nt(N-1)(1 - t)^{N-2} - N(1 - t)^{N-1} \\
    &= - Nt(N-1)(1 - t)^{N-2} < 0.
  \end{align*}
  Thus, $\phi$ is decreasing, and since $\lim_{t \to 0_+}\phi(t) = 0$, we have $\phi(t) < 0$ for all $t > 0$. This implies $h'(t) < 0$, so $h$ is decreasing on $(0,1/N]$. Therefore, for $t\in(0,1/N]$,
  \begin{align}\label{smalllower}
  h(t)=\frac{1 - (1-t)^N}{Nt} \ge h\left(\frac{1}{N}\right) = 1 - \left(1 - \frac{1}{N}\right)^N \ge 1 - \frac{1}{e}.
  \end{align}
  Combining \eqref{small00}, \eqref{smallupper}, and \eqref{smalllower}, for any $t\in(0,1/N]$, we have
\begin{align}\label{small}
  \left(1 - e^{-1}\right)  \min\left\{ 1, \frac{1}{Nt} \right\}\le \frac{1 - (1-t)^N}{Nt} \le  \min\left\{ 1, \frac{1}{Nt} \right\}.
\end{align}

\noindent\textbf{The case $1/N < t \le 1$.} In this case,
  \begin{align}\label{large00}
  \min\left\{ 1, \frac{1}{Nt} \right\} = \frac{1}{Nt}.
  \end{align}
Recall $N\geq 1$ and $0\leq (1-t)^N \leq 1$, so
  \begin{align}\label{largeupper}
  \frac{1 - (1-t)^N}{Nt} \le \frac{1}{Nt}.
  \end{align}
  For the lower bound, since $t > 1/N$, we have 
  \[(1-t)^N < \left(1-\frac{1}{N}\right)^N \le \frac{1}{e}.\]
  Hence, \[ 1 - (1-t)^N \ge 1 - \frac{1}{e}. \]
  Therefore,
  \begin{align}\label{largelower}
  \frac{1 - (1-t)^N}{Nt} \ge \left(1 - \frac{1}{e}\right)\frac{1}{Nt}.
  \end{align}
  Combining \eqref{largelower}, \eqref{largeupper}, and \eqref{large00}, for all $t \in (1/N, 1]$:
  \begin{align}\label{large}
  \left(1 - e^{-1}\right) \min\left\{ 1, \frac{1}{Nt} \right\}\le \frac{1 - (1-t)^N}{Nt} \le  \min\left\{ 1, \frac{1}{Nt} \right\}.
  \end{align}

  From \eqref{small} and \eqref{large}, for all $t\in (0,1]$, we have
  \begin{align*}
  \left(1 - e^{-1}\right) \min\left\{ 1, \frac{1}{Nt} \right\}\le \frac{1 - (1-t)^N}{Nt} \le  \min\left\{ 1, \frac{1}{Nt} \right\}.
  \end{align*}
  Integrating both sides with respect to $\mu$ over $[0,1]$, we obtain the two-sided inequality \eqref{claim111}. This completes the verification of Claim \ref{claim1}.
\end{proof}
We then complete the proof of Lemma \ref{mnestmate1}.
\end{proof}

Now, we are ready to prove Proposition \ref{prop:mn_limsup_estimate}. Recall that our goal is to show
\begin{align*}
\limsup_{n \to \infty} \frac{\log m_n}{\log(n+1)} = -s_0,
\end{align*}
where
\[
m_n = \frac{1}{n+1} \int_0^1 \frac{1 - r^{n+1}}{1 - r} d\nu(r), \quad n \geq 0.
\]
where $s_0$ is defined by \eqref{eq:s0_def}:
\begin{align*}
s_0 = \sup\left\{ 0 \le s < 1 \,:\, \int_0^1 (1 - r)^{-s} \, d\nu(r) < \infty \right\}.
\end{align*}

\begin{proof}
Let
\[
L = \limsup_{n \to \infty} \frac{\log m_n}{\log(n+1)}.
\]
Observe that for any non-negative integer $n$ and $r \in [0,1]$,
$
1-r^{n+1}\geq 1-r.
$
Then
\begin{align*}
L= \limsup_{n\to\infty} \frac{\log m_n}{\log (n+1)} \geq \limsup_{n\to\infty} \frac{\log\left(\frac{1}{n + 1}\right) + \log\nu\left([0,1]\right)}{\log(n + 1)} = -1,
\end{align*}
and we have
\begin{align*}
L\geq -1.
\end{align*}
Next, we establish the upper bound
\begin{align}\label{upper-L}
L \leq -s_0.
\end{align}
We proceed by contradiction. Suppose $L > -s_0$. Then there exists a real number $t$ such that
\[
L > t > -s_0.
\]
By the definition of the limit superior, there exists a subsequence $\{n_k\}_{k\in\mathbb{N}}$ and a constant $K_0 > 0$ such that for all $k > K_0$,
\[
\frac{\log m_{n_k}}{\log(n_k+1)} > t.
\]
Exponentiating both sides (noting $\log(n_k+1) > 0$ for all $k$), we obtain
\[
m_{n_k} > (n_k+1)^t.
\]
Choose $s \in \mathbb{R}$ such that $-t < s < s_0$. By Lemma \ref{mnestmate1}, the series converges:
\[
\sum_{n=0}^\infty m_n (n+1)^{s-1} < \infty.
\]
In particular, the tail sums are uniformly bounded for all $k \geq K_0$:
\begin{align}\label{block-sum}
\sup_{k \geq K_0} \sum_{n=n_k}^{2n_k} m_n (n+1)^{s-1} \leq \sum_{n=0}^\infty m_n (n+1)^{s-1} < \infty.
\end{align}
We now derive a contradiction by bounding the sum over $[n_k, 2n_k]$ from below. For any $n \in [n_k, 2n_k]$, we use the definition of $m_n$ to get:
\begin{align*}
m_n &= \frac{1}{n+1} \int_0^1 \frac{1 - r^{n+1}}{1 - r} d\nu(r) \geq \frac{1}{n+1} \int_0^1 \frac{1 - r^{n_k+1}}{1 - r} d\nu(r)\\
&= \frac{n_k+1}{n+1} m_{n_k} \geq \frac{1}{2} (n_k+1)^t. 
\end{align*}
Substituting this lower bound into the sum over $[n_k, 2n_k]$:
\begin{align*}
\sum_{n=n_k}^{2n_k} m_n (n+1)^{s-1}
&\geq \sum_{n=n_k}^{2n_k} \frac{1}{2} (n_k+1)^t (n+1)^{s-1} \\
&\geq \frac{1}{2} (n_k+1)^t (2n_k+2)^{s-1} (n_k + 1) \\
&= 2^{s - 2}\cdot (n_k + 1)^{t + s}.
\end{align*}
By our choice of $s$, we have $t + s > 0$, so
\[
\lim_{k \to \infty} (n_k+1)^{t + s} = \infty.
\]
This implies
\[
\lim_{k \to \infty} \sum_{n=n_k}^{2n_k} m_n (n+1)^{s-1} = \infty,
\]
which contradicts the uniform boundedness in \eqref{block-sum}. Thus the inequality \eqref{upper-L} follows.

Now, we prove the following inequality:
\begin{align}\label{Ls0}
L \geq -s_0.
\end{align}
We proceed by contradiction and suppose that
\[
-1 \leq L < -s_0<1.
\]
Then, there exists a real number \( \delta \) such that
\[
-1 \leq L < \delta < -s_0.
\]
Moreover, there exists a positive integer \( N \) such that for all integers \( n \geq N \),
\[
\frac{\log m_n}{\log(n+1)} < \delta.
\]
Then, we obtain an upper bound for \( m_n \):
\begin{align}\label{mn-upper}
m_n < (n+1)^\delta, \quad n \geq N.
\end{align}
Now choose a real number \( s \) satisfying
\[
s_0 < s < -\delta < 1.
\]
Since \( s_0<s <1 \), by the definition of \( s_0 \), we have
\begin{align}\label{div}
\int_0^1 \frac{1}{(1 - r)^s} \, d\nu(r) = \infty.
\end{align}
By Lemma \ref{mnestmate1}, this integral divergence implies
\[
\sum_{n=0}^\infty m_n (n+1)^{s-1} = \infty.
\]
By the upper bound \eqref{mn-upper}, we estimate the tail of the series:
\[
\sum_{n=N}^{\infty} m_n (n+1)^{s-1} < \sum_{n=N}^{\infty} (n+1)^\delta (n+1)^{s-1} = \sum_{n=N}^{\infty} (n+1)^{\delta + s - 1}.
\]
By our choice of \( s < -\delta \), we have \( \delta + s < 0 \), so the exponent satisfies
\[
\delta + s - 1 < -1.
\]
It follows that
\[
\sum_{n=0}^{\infty} m_n (n+1)^{s-1} < \infty,
\]
which contradicts the divergence result \eqref{div}. This contradiction invalidates our initial assumption, so we conclude the desired inequality \eqref{Ls0}.

By inequalities \eqref{upper-L} and \eqref{Ls0},  we complete the proof of Proposition \ref{prop:mn_limsup_estimate}.
\end{proof}

\begin{corollary}\label{prop:sum_of_mn}
Let $s_0$ be defined as in \eqref{eq:s0_def}.
If $0 < s_0 < 1$, for any $0 < \varepsilon < 1 - s_0$, there exists a strictly increasing sequence \( \{n_k\}_{k=1}^\infty \) with \( m_{n_k} \ge (n_k + 1)^{-s_0 - \varepsilon} \) such that for all \( s > s_0 \),
\[
\sum_{k = 1}^{\infty} m_{n_k} (n_k + 1)^{s - 1} = \infty.
\]  
\end{corollary}

\begin{proof}

  We proceed by contradiction. For any $0 < \varepsilon < 1 - s_0$, by Proposition \ref{prop:mn_limsup_estimate} and the definition of limit superior, the set $$\Lambda = \{ n : m_n \ge (n+1)^{-s_0 - \varepsilon}\}$$ is infinite,
  suppose $s_0 < s \le s_0 + \frac{\varepsilon}{2}$ such that
  \[
  \sum_{n \in \Lambda} m_n (n+1)^{s - 1} = C < \infty.
  \]
  Observe that
  \[
  \sum_{n=0}^{\infty} m_n (n+1)^{s - 1} 
    = \sum_{n \in \Lambda} m_n (n+1)^{s - 1} + \sum_{n \notin \Lambda} m_n (n+1)^{s - 1} .
  \]
  Then, we have
  \begin{align*}
    \sum_{n=0}^{\infty} m_n (n+1)^{s - 1} 
    &\le \ C + \sum_{n \notin \Lambda} (n+1)^{-s_0 - \varepsilon} (n+1)^{s - 1} \\
    &= C + \sum_{n \notin \Lambda} (n+1)^{s - s_0 - \varepsilon - 1} \le C + \sum_{n=0}^{\infty} (n+1)^{-\frac{\varepsilon}{2} - 1} < \infty.
  \end{align*}
  By Lemma \ref{mnestmate1}, this is a contradiction for the definition of $s_0$. Therefore, we have
  \[
  \sum_{n \in \Lambda} m_n (n+1)^{s - 1} = \infty, \quad s_0 < s \le s_0 + \frac{\varepsilon}{2}.
  \]
  By monotonicity of $\sum_{n \in \Lambda} m_{n_k}(n_k + 1)^{s - 1}$ in $s$. Then for $s > s_0 + \frac{\varepsilon}{2}$, we obtain 
  \[
  \sum_{n \in \Lambda} m_n (n+1)^{s - 1} = \infty.
  \]
  This completes the proof.
\end{proof}

\section{The $\alpha$-Bergman-CZOs on the unit disk}
For the endpoint estimates needed in this paper, we recall some results on $\alpha$-Bergman-CZOs on the unit disk.
One will note that, in the study of the Shimorin-type operator $T_\nu$, the related $\alpha$ is completely determined by $c_\nu$.

\begin{definition}
Let $\alpha\in[1,2]$. A kernel $K:\mathbb{D}\times\mathbb{D}\to\mathbb{C}$, which is analytic in the first variable and anti-analytic in the second variable, is said to be an $\alpha$-Bergman--Calderón--Zygmund kernel on the unit disk $\mathbb{D}$ if there exists a positive constant $C$ such that the following two conditions hold:
\begin{align}\label{kernel-size}
|K(z,\lambda)|\leq \frac{C}{|1-z\overline{\lambda}|^\alpha}
\end{align}
and
\begin{align}\label{kernel-smooth}
|\partial_z K(z,\lambda)|+|\partial_{\bar{\lambda}} K(z,\lambda)|\leq \frac{C}{|1-z\overline{\lambda}|^{\alpha+1}}.
\end{align}
Moreover, an integral operator
\[
T_K f(z)=\int_\mathbb{D}K(z,\lambda)f(\lambda)dA(\lambda)
\]
is called an $\alpha$-Bergman-Calderón-Zygmund operator ($\alpha$-Bergman-CZO) if its kernel $K$ is an $\alpha$-Bergman--Calderón--Zygmund kernel.
\end{definition}

For \( 1\leq p < \infty \), recall that the weak \( L^{p,\infty}(\mathbb{D}) \) space is defined as the collection of all measurable functions \( f \) for which the weak \( L^p \) quasi-norm
\[
\| f \|_{L^{p,\infty}(\mathbb{D})} = \inf \left\{ C > 0 : d_f(t) \leq \frac{C^p}{t^p}, \quad \text{for all } \ t > 0 \right\}
\]
is finite. Here, \( d_f(t) \) denotes the distribution function of \( f \), given by
\[
d_f(t) = |\{ z \in \mathbb{D} : |f(z)| > t \}|,
\]
where $|\cdot|$ is the normalized area measure of a measurable set.
%and \( \mathbf{1}_A \) is the indicator function of a set \( A \): explicitly, \( \mathbf{1}_A(z) = 1 \) if \( z \in A \), and \( \mathbf{1}_A(z) = 0 \) otherwise.

We next recall the definition of the space $\operatorname{BMO}(\mathbb{D})$, which consists of all measurable functions $f: \mathbb{D} \to \mathbb{C}$ satisfying
\[
\| f \|_{\operatorname{BMO}(\mathbb{D})} = \sup_{Q \subset \mathbb{D}} E_Q\left(|f - E_Q(f)|\right) < \infty,
\]
where the supremum is taken over all Euclidean cubes $Q$ in the unit disk $\mathbb{D}$. Here, $E_Q(f)$ denotes the average value of $f$ over $Q$, defined as
\[
E_Q(f) = \frac{1}{|Q|} \int_Q f(z) \, dA(z),
\]
and $|Q|$ is the normalized Lebesgue measure of the cube $Q$. 

\begin{remark}\label{BMO-bloch}
In the definition of $\operatorname{BMO}(\mathbb{D})$, the cubes $Q \subset \mathbb{D}$ can be replaced by any homogeneous balls in $\mathbb{D}$ (see \cite[Propositions 24 and 25]{FGW}). Recall that the Bloch space $\mathcal{B}$ is the space of holomorphic functions $f$ on $\mathbb{D}$ satisfying
\[
\| f \|_{\mathcal{B}} = |f(0)| + \sup_{z \in \mathbb{D}} (1 - |z|^2) |f'(z)| < \infty,
\]
where $\| \cdot \|_{\mathcal{B}}$ denotes the Bloch norm. A classical theorem due to Coifman–Rochberg–Weiss \cite[pp. 632]{CRW} states that there exist positive constants $C_1, C_2 > 0$ such that
\[
C_1 \| f \|_{\mathcal{B}} \leq \| f \|_{\operatorname{BMO}(\mathbb{D})} \leq C_2 \| f \|_{\mathcal{B}}
\]
for every $f \in \mathcal{B}$. Note that \( T_K f \) is holomorphic on the unit disk. Hence, 
for the integral operator $T_K$ the BMO-type estimate is equivalent to a Bloch estimate.
\end{remark}

\begin{proposition}\label{alpha Bergman CZO}
Let $\alpha\in[1,2]$. Suppose that $T_K$ is an $\alpha$-Bergman--Calderón--Zygmund operator. Then
\begin{itemize}
    \item [(a)] If $1 < p < \frac{2}{2 - \alpha}$ and $1/q = 1/p + \alpha/2 - 1$, then $T_{K}$ is bounded from $L^p(\mathbb{D})$ to $L^q(\mathbb{D})$;
    \item [(b)] $T_{K}$ is bounded from $L^1(\mathbb{D})$ to $L^{\frac{2}{\alpha},\infty}(\mathbb{D})$;
    \item [(c)] $T_{K}$ is bounded from $L^{\frac{2}{2 - \alpha}}(\mathbb{D})$ to $\mathcal{B}$. 
\end{itemize}
\end{proposition}

\begin{proof}[Proof of Proposition \ref{alpha Bergman CZO}]
If $\alpha = 2$, the proof of (a) and (b) in Proposition \ref{alpha Bergman CZO} follows from the results of classical Bergman projections (see \cite[Theorem 1.1]{DHZZ} and \cite{Zhu14}).
Let $\alpha\in[1,2)$. By \eqref{kernel-size}, we have
\[
T_Kf(z)\leq C\int_{\mathbb{D}}\frac{|f(\lambda)|}{|1-z\overline{\lambda}|^\alpha}dA(\lambda)
\leq C\int_{\mathbb{D}}\frac{|f(\lambda)|}{|z-\lambda|^\alpha}dA(\lambda).
\]
Hence, by the Hardy–Littlewood–Sobolev theorem (see \cite[Chapter 6]{Gra}), we obtain (a) and (b).
For the proof of $(c)$, observe that for $f \in L^{\frac{2}{2 - \alpha}}(\mathbb{D})$, we have 
  \[
  (T_{K}f)'(z) = \partial_z \int_{\mathbb{D}} K(z,\lambda)f(\lambda)dA(\lambda) = \int_{\mathbb{D}}\partial_{z} K(z,\lambda) f(\lambda)dA(\lambda).
  \]
By the smoothness condition \eqref{kernel-smooth}, there exists a constant $C$ such that for any $z, \lambda \in\mathbb{D}$
  \[
  |\partial_z K(z,\lambda)|\leq \frac{C}{|1-z\overline{\lambda}|^{\alpha+1}}.
  \]
  Recall the Forelli--Rudin type estimate \cite[Lemma 3.10]{Zhu14}: for $t > -1$ and $c > 0$, we have
  \begin{equation}\label{eq:FR}
    \int_{\mathbb{D}} \frac{(1-|\lambda|^2)^t}{|1-z\overline{\lambda}|^{2+c+t}} dA(\lambda)
    \le C_{t,c}\,(1-|z|^2)^{-c},\quad z\in\mathbb{D}.
  \end{equation}
  Thus, we have
  \begin{align*}
    |(T_{K}f)'(z)| 
    &= \left |\int_{\mathbb{D}}\partial_{z} K(z,\lambda) f(\lambda)dA(\lambda) \right |\\
    &\le C \int_{\mathbb{D}} \frac{1}{|1 - z\overline{\lambda}|^{\alpha + 1}}|f(\lambda)|dA(\lambda). 
  \end{align*} 
  By H\"older's inequality and take $t = 0, c = \alpha$ in \eqref{eq:FR},
  \begin{align*}
    \int_{\mathbb{D}} \frac{1}{|1 - z\overline{\lambda}|^{\alpha + 1}}|f(\lambda)|dA(\lambda) 
    &\le  \left( \int_{\mathbb{D}} \frac{1}{|1 - z\overline{\lambda}|^{2 + 2/\alpha}}dA(\lambda)\right)^{\alpha/2} \|f\|_{L^{\frac{2}{2 - \alpha}}(\mathbb{D})}\\
    &\le  C_{0,\alpha} \frac{1}{1 - |z|^2}\|f\|_{L^{\frac{2}{2 - \alpha}}(\mathbb{D})}.
  \end{align*}
  Hence, we have
  \[
  \sup_{z \in \mathbb{D}} (1 - |z|^2)|(T_{K}f)'(z)| \le C C_{0,\alpha} \|f\|_{L^{\frac{2}{2 - \alpha}}(\mathbb{D})}.
  \]
  Moreover,
  \[
  |T_{K}f(0)| \le \int_{\mathbb{D}} |K(0,\lambda)| |f(\lambda)| dA(\lambda) \le C \|f\|_{L^{\frac{2}{2 - \alpha}}(\mathbb{D})} \left( \int_{\mathbb{D}} dA(\lambda) \right)^{\alpha/2} \le C \|f\|_{L^{\frac{2}{2 - \alpha}}(\mathbb{D})}.
  \]
  Combining these two estimates, we obtain
  \[\|T_{K}f\|_{\mathcal{B}} \le (C C_{0,\alpha} + C) \|f\|_{L^{\frac{2}{2 - \alpha}}(\mathbb{D})}.\]
  Hence, we conclude that $T_{K}: L^{\frac{2}{2 - \alpha}}(\mathbb{D}) \to \mathcal{B}$ is bounded.
  This completes the whole proof.
\end{proof}

Given a finite positive Borel measure $\nu$ on $[0,1]$, recall the quantity $c_\nu$ defined by \eqref{critical-index},
\begin{align*}
c_\nu = \sup\left\{1 \le c < 2: \int_{0}^{1} \frac{1}{(1-r^2)^{2/c'}} \,d\nu(r) < \infty\right\}.
\end{align*}
\begin{lemma}\label{u-estimate of K}
Let $\nu$ be a finite positive Borel measure on $[0,1]$ such that $1 \le c_{\nu} < 2$ and $\nu(\{1\}) = 0$. Suppose that $\nu$ satisfies the 
$2/c_\nu'$-Carleson condition, namely,
\begin{equation}\label{2/c-Carleson condition}
C_{c_\nu,\nu} = \sup_{0 < t \le 1} \frac{\nu([1 - t, 1))}{t^{2 - 2/c_\nu}} < \infty.
\end{equation}
Then the Shimorin-type operator $T_\nu$ is a $2/c_\nu$-Bergman--Calderón--Zygmund operator.
\end{lemma}

\begin{corollary}\label{cor:forcnu1}
Let $\nu$ be a finite positive Borel measure on $[0,1]$ with $c_{\nu}=1$. Then the Shimorin-type operator $T_\nu$ is a $2$-Bergman--CZO on the unit disk.
\end{corollary}
\begin{proof}
When $\nu(\{1\})=0$, the statement is a direct consequence of Lemma \ref{u-estimate of K}.  
If $\nu(\{1\})\ne0$, decompose $\nu$ as  
\[
\nu = \nu_1 + \nu(\{1\})\delta_1,
\]  
where $\nu_1(I)=\nu\big([0,1)\cap I\big)$ for any Borel set $I\subset [0,1]$.  
Then $\nu_1$ is a finite positive Borel measure on $[0,1]$ with $\nu_1(\{1\})=0$, and the kernel splits as  
\[
K_\nu(z,\lambda)=K_{\nu_1}(z,\lambda)+\frac{\nu(\{1\})}{(1-z\overline{\lambda})^2}.
\]  
Hence  
\[
T_{K_\nu}=T_{K_{\nu_1}}+\nu(\{1\})P,
\]  
with $P$ the Bergman projection. By Lemma \ref{u-estimate of K}, $T_{\nu_1}$ is a $2/c_{\nu_1}$-Bergman–CZO (in particular a $2$-Bergman–CZO), while $P$ is a standard $2$-Bergman–CZO. Therefore $T_{K_\nu}$ is also a $2$-Bergman–CZO.
This completes the whole proof.
\end{proof}
\begin{proof}[Proof of Lemma \ref{u-estimate of K}]
Let $\alpha = 2/c_\nu$. We aim to show that $K_\nu$ satisfies the conditions \eqref{kernel-size} and \eqref{kernel-smooth} for this $\alpha$. 

\noindent\textbf{The case $c_\nu = 1$.} 
In this case, $\alpha = 2$. 
For the size estimate by Lemma \ref{zw estimate}, we have
\begin{align}\label{universal_estimates}
|K_\nu(z,\lambda)| &\le \frac{1}{|1-z\overline{\lambda}|} \int_0^1 \frac{1}{|1-rz\overline{\lambda}|} d\nu(r) \le  \frac{2\nu([0,1])}{|1-z\overline{\lambda}|^2}.
\end{align}
For the smooth estimate, since $K_\nu(z,\lambda)=\overline{K_\nu(\lambda,z)}$, it suffices to show 
\[|\partial_z K(z,\lambda)|\leq \frac{C}{|1-z\overline{\lambda}|^{\alpha+1}}\]
for a positive constant $C$.
Differentiating the kernel yields
\[
\partial_{z}K_{\nu}(z,\lambda) = \frac{\overline{\lambda}}{(1 - z\overline{\lambda})^2} \int_{0}^{1} \frac{1}{1 - rz\overline{\lambda}}d\nu(r) + \frac{\overline{\lambda}}{1 - z\overline{\lambda}} \int_{0}^{1} \frac{r}{(1 - rz\overline{\lambda})^2}d\nu(r).
\]
Taking absolute values and using $|r|\le1$ and $|\lambda| < 1$, we obtain
\begin{equation}\label{eq: partial K}
\begin{aligned}
|\partial_{z}K_{\nu}(z,\lambda)|
&\leq \frac{|\lambda|}{|1 - z\overline{\lambda}|^2} \int_{0}^{1} \frac{1}{|1 - rz\overline{\lambda}|}d\nu(r) + \frac{|\lambda|}{|1 - z\overline{\lambda}|} \int_{0}^{1} \frac{r}{|1 - rz\overline{\lambda}|^2}d\nu(r) \\
&\leq \frac{1}{|1 - z\overline{\lambda}|^2} \int_{0}^{1} \frac{1}{|1 - rz\overline{\lambda}|}d\nu(r) + \frac{1}{|1 - z\overline{\lambda}|} \int_{0}^{1} \frac{1}{|1 - rz\overline{\lambda}|^2}d\nu(r).
\end{aligned}
\end{equation}
Using Lemma \ref{zw estimate} again and combining the estimate \eqref{eq: partial K}, we obtain
\[
|\partial_{z}K_{\nu}(z,\lambda)| \le \frac{2\nu([0,1])}{|1-z\overline{\lambda}|^3} + \frac{4\nu([0,1])}{|1-z\overline{\lambda}|^3} = \frac{6\nu([0,1])}{|1-z\overline{\lambda}|^3}.
\]
Thus, $K_\nu$ is a $2$-Bergman--CZO.

\noindent\textbf{The case $1 < c_\nu < 2$.} Recall the definition of $c_{\nu}$ \eqref{critical-index}. In this case, we have $\nu(\{1\}) = 0$. Then $\nu([1-t,1]) = \nu([1-t,1))$ for each $0 < t \le 1$. Then
\[
C_{c_\nu,\nu} = \sup_{0 < t \le 1} \frac{\nu([1 - t, 1))}{t^{2 - 2/c_\nu}} = \sup_{0 < t \le 1} \frac{\nu([1 - t, 1])}{t^{2 - 2/c_\nu}}.
\]
By Lemma \ref{zw estimate} we have the pointwise bound
\begin{equation}\label{pointwise_lower_bound}
|1-rz\overline{\lambda}|\ge \max\left\{\frac12|1-z\overline{\lambda}|,1-r\right\}.
\end{equation}
Splitting the integral accordingly, we obtain
\begin{align*}
\int_{0}^{1}\frac{1}{|1-rz\overline{\lambda}|}\,d\nu(r)
&=\int_{0}^{1-\frac12|1-z\overline{\lambda}|}\frac{1}{|1-rz\overline{\lambda}|}\,d\nu(r)
+\int_{1-\frac12|1-z\overline{\lambda}|}^{1}\frac{1}{|1-rz\overline{\lambda}|}\,d\nu(r)\\
&\le \int_{0}^{1-\frac12|1-z\overline{\lambda}|}\frac{1}{1-r}\,d\nu(r)
+\frac{2}{|1-z\overline{\lambda}|}\,\nu\!\left(\left[1-\frac12|1-z\overline{\lambda}|,\,1\right]\right)\\
&=: I_1+I_2.
\end{align*}
Defining $\Phi(u) = \nu([1-u, 1])$ and using the change-of-variable formula with $u = 1 - r$, we have
\[
	I_1 = \int_{\frac{1}{2}|1-z\overline{\lambda}|}^1\frac{1}{u}d\Phi(u).
\]
By the Stieltjes integral formula with $\nu(\{1\})=0$, we obtain
\[
I_1 = \left. \frac{\Phi(u)}{u} \right|_{\frac{1}{2}|1-z\overline{\lambda}|}^1 + \int_{\frac{1}{2}|1-z\overline{\lambda}|}^1 \frac{\Phi(u)}{u^2} du = \nu([0,1]) - \frac{\Phi(\frac{1}{2}|1-z\overline{\lambda}|)}{\frac{1}{2}|1-z\overline{\lambda}|} + \int_{\frac{1}{2}|1-z\overline{\lambda}|}^1 \frac{\Phi(u)}{u^2} du.
\]
The term 
\[
- \frac{\Phi(\frac{1}{2}|1-z\overline{\lambda}|)}{\frac{1}{2}|1-z\overline{\lambda}|} = - \frac{2}{|1-z\overline{\lambda}|} \nu\left(\left[1-\frac{1}{2}|1-z\overline{\lambda}|, 1\right]\right)
\]
 exactly cancels $I_2$. 
By the $2/c_\nu'$-Carleson condition \eqref{2/c-Carleson condition} we have $\Phi(u) \le C_{c_\nu,\nu} u^{2-2/c_\nu}$, in particular, $\nu([0,1]) = \Phi(1) \le C_{c_\nu,\nu}$, we get
\[
	 \int_{0}^{1}\frac{1}{|1-rz\overline{\lambda}|} d\nu(r) \le \nu([0,1]) + C_{c_\nu,\nu} \int_{\frac{1}{2}|1-z\overline{\lambda}|}^1 u^{-2/c_\nu} du.
\]
Substituting the integral result gives
\begin{align*}
  \int_{0}^{1}\frac{1}{|1-rz\overline{\lambda}|} d\nu(r) 
  &\le C_{c_\nu,\nu} + C_{c_\nu,\nu}\frac{c_\nu}{c_\nu-2}\left(1-\left(\frac{1}{2} |1-z\overline{\lambda}|\right)^{1-2/c_\nu}\right)\\
  &\le C_{c_\nu,\nu}\left(\frac{2c_\nu-2}{c_\nu-2} \right) + C_{c_\nu,\nu} \frac{c_\nu}{2-c_\nu} \left(\frac{1}{2}|1-z\overline{\lambda}|\right)^{1-2/c_\nu}.
\end{align*}
Since $1 < c_\nu < 2$, we have $\frac{2c_\nu-2}{c_\nu-2}\leq 0$.
Hence,
\begin{equation}\label{eq: second part of K}
	\int_{0}^{1}\frac{1}{|1-rz\overline{\lambda}|} d\nu(r) \leq   C_{c_\nu,\nu}\frac{c_\nu \,2^{2/c_\nu-1}}{2-c_\nu} \frac{1}{|1-z\overline{\lambda}|^{2/c_\nu-1} }.
\end{equation}
Combining this with the factor $|1-z\overline{\lambda}|^{-1}$, we obtain
\begin{align*}
  |K_\nu(z,\lambda)| \le  C_{c_\nu,\nu} \frac{c_\nu2^{2/c_\nu-1}}{2-c_\nu}\frac{1}{|1-z\overline{\lambda}|^{2/c_\nu} }.
\end{align*}

	For the smooth estimate, by \eqref{eq: partial K} we have
\begin{align*}
  |\partial_{z}K_{\nu}(z,\lambda)| 
  &\le \frac{1}{|1 - z\overline{\lambda}|^2} \int_{0}^{1} \frac{1}{|1 - rz\overline{\lambda}|}d\nu(r) + \frac{1}{|1 - z\overline{\lambda}|} \int_{0}^{1} \frac{1}{|1 - rz\overline{\lambda}|^2}d\nu(r)\\
  &=: J_1 + J_2.
\end{align*}
$J_1$ is bounded by the estimate \eqref{eq: second part of K} as 
\begin{equation}\label{J_1}
J_1 \le C_{c_\nu,\nu} \frac{c_\nu2^{2/c_\nu-1}}{2-c_\nu} \frac{1}{|1-z\overline{\lambda}|^{2/c_\nu + 1} }.
\end{equation}
Recall that $\Phi(u) = \nu([1-u, 1])$.
By the estimate \eqref{pointwise_lower_bound}, we obtain
\begin{align*}
\int_{0}^{1}\frac{1}{|1 - rz\overline{\lambda}|^2}d\nu(r) 
&\leq\int_{0}^{1-\frac12|1-z\overline{\lambda}|}\frac{1}{(1-r)^2} d\nu(r)+\int_{1-\frac12|1-z\overline{\lambda}|}^{1}\frac{4}{|1-z\overline{\lambda}|^2}\,d\nu(r)\\
&\le \int_{\frac{1}{2}|1-z\overline{\lambda}|}^1 \frac{1}{u^2} d\Phi(u) + \frac{4}{|1-z\overline{\lambda}|^2} \Phi\left(\frac{|1-z\overline{\lambda}|}{2}\right).
\end{align*}
Using the Stieltjes integral formula gives
\begin{align*}
\int_{0}^{1}\frac{1}{|1 - rz\overline{\lambda}|^2}d\nu(r) 
&\leq \nu([0,1]) - \frac{\Phi\left(\frac{1}{2}|1-z\overline{\lambda}|\right)}{\left(\frac{1}{2}|1-z\overline{\lambda}|\right)^2} + 2\int_{\frac{1}{2}|1-z\overline{\lambda}|}^1 \frac{\Phi(u)}{u^3} du +  \frac{\Phi\left(\frac{1}{2}|1-z\overline{\lambda}|\right)}{\left(\frac{1}{2}|1-z\overline{\lambda}|\right)^2}\\
&= \nu([0,1]) + 2\int_{\frac{1}{2}|1-z\overline{\lambda}|}^1 \frac{\Phi(u)}{u^3} du.
\end{align*}
Since $\Phi(u) \le C_{c_\nu,\nu} u^{2-2/c_\nu}$, the integral term yields
\[
\int_{0}^{1}\frac{1}{|1 - rz\overline{\lambda}|^2}d\nu(r) \le \nu([0,1]) + 2C_{c_\nu,\nu} \int_{\frac{1}{2}|1-z\overline{\lambda}|}^1 u^{-1-2/c_\nu} du.
\]
Substituting the integral result gives
\begin{align*}
	\int_{0}^{1}\frac{1}{|1 - rz\overline{\lambda}|^2}d\nu(r)
	&\le C_{c_\nu,\nu} - C_{c_\nu,\nu} c_{\nu} \left(1-\left(\frac{1}{2}|1-z\overline{\lambda}|\right)^{-2/c_\nu}\right)  \\
	&\le C_{c_\nu,\nu} (1-c_\nu) + C_{c_\nu,\nu} c_{\nu} \left(\frac{1}{2}|1-z\overline{\lambda}|\right)^{-2/c_\nu}.
\end{align*}
Since $1 - c_\nu < 0$, we have
\begin{equation}\label{eq: part of partial K}
	\int_{0}^{1}\frac{1}{|1 - rz\overline{\lambda}|^2}d\nu(r)\leq 2^{2/c_\nu} c_\nu C_{c_\nu,\nu}\frac{1}{|1-z\overline{\lambda}|^{2/c_\nu}}.
\end{equation}
Multiplying by $|1-z\overline{\lambda}|^{-1}$, then
\begin{equation}\label{J_2}
J_2 \le C_{c_\nu,\nu} c_\nu 2^{2/c_\nu}\frac{1}{|1-z\overline{\lambda}|^{1+2/c_\nu}}.
\end{equation}
Combining the estimates for \eqref{J_1} and \eqref{J_2}, we obtain
\[
|\partial_z K_\nu(z,\lambda)| \le C_{c_\nu,\nu}c_\nu 2^{2/c_\nu} \left( \frac{1}{2(2-c_\nu)} + 1 \right) \frac{1}{|1-z\overline{\lambda}|^{1+2/c_\nu}}.
\]
This confirms that $K_\nu$ is a $2/c_\nu$-Bergman--CZO for $1 < c_\nu < 2$.
\end{proof}

\begin{lemma}\label{u-estimate of K-2}
Let $\nu$ be a finite positive Borel measure on $[0,1]$ with $c_{\nu}=2$.
Suppose that $\nu$ is integrable with respect to the hyperbolic length on $[0,1)$. That is
\[
\int_{[0,1)}\frac{1}{1-r^2}d\nu(r)<\infty.
\]
Then the Shimorin-type operator $T_\nu$ is a $1$-Bergman--Calderón--Zygmund operator.
\end{lemma}

\begin{proof}
Since $c_{\nu}=2$, we have $\nu(\{1\})=0$.
Note that the hyperbolic integrability of $\nu$ on $[0,1)$ is equivalent (up to absolute constants) to 
\begin{equation*}
C :=\int_0^1\frac{1}{1-r}d\nu(r)<\infty.
\end{equation*}
Hence, for $0 < t \le 1$, we have
\[
\frac{\nu([1 - t, 1])}{t} = \int_{1 - t}^1 \frac{1}{t}\,d\nu(r) \le \int_{1 - t}^1 \frac{1}{1 - r}\,d\nu(r) < C.
\]
Consequently, we obtain
\[
C_{2,\nu}:= \sup_{0 < t \le 1}\frac{\nu([1-t,1))}{t} =  \sup_{0 < t \le 1}\frac{\nu([1-t,1])}{t} < C.
\]
By the estimate \eqref{pointwise_lower_bound}, we estimate the integral as follows:
\begin{equation}\label{1estimate}
\begin{aligned}
\int_{0}^{1} \frac{1}{|1 - rz\overline{\lambda}|}\,d\nu(r) 
&\le \int_{0}^{1 - \frac{1}{2}|1 - z\overline{\lambda}|} \frac{1}{1 - r}\,d\nu(r) + \int_{1 - \frac{1}{2}|1 - z\overline{\lambda}|}^1 \frac{2}{|1 - z\overline{\lambda}|}\,d\nu(r)\\
&\le \int_{0}^{1} \frac{1}{1 - r}\,d\nu(r) + \frac{2}{|1 - z\overline{\lambda}|}\nu\!\left(\left[1-\frac{1}{2}|1-z\overline{\lambda}|,\,1\right]\right)\\
&\le C + \frac{2}{|1 - z\overline{\lambda}|}\,C_{2,\nu}\,\frac{1}{2}|1-z\overline{\lambda}|\\
&\le C + C_{2,\nu}.
\end{aligned}
\end{equation}
Simplifying the right-hand side, we get
\begin{align*}
|K_{\nu}(z,\lambda)|
&= \left|\frac{1}{1-z\overline{\lambda}}\int_{0}^{1}\frac{1}{1-rz\overline{\lambda}}\,d\nu(r)\right|\\
&\le \frac{1}{|1-z\overline{\lambda}|}\int_{0}^{1}\frac{1}{|1-rz\overline{\lambda}|}\,d\nu(r)
\le (C + C_{2,\nu})\,\frac{1}{|1-z\overline{\lambda}|}.
\end{align*}

Next, we verify the smoothness condition. By \eqref{eq: partial K},
\[
|\partial_{z}K_{\nu}(z,\lambda)| \le \frac{1}{|1 - z\overline{\lambda}|^2} \int_{0}^{1} \frac{1}{|1 - rz\overline{\lambda}|}\,d\nu(r) + \frac{1}{|1 - z\overline{\lambda}|} \int_{0}^{1} \frac{1}{|1 - rz\overline{\lambda}|^2}\,d\nu(r).
\]
From \eqref{1estimate}, we know
\[
\int_{0}^{1} \frac{1}{|1 - rz\overline{\lambda}|}\,d\nu(r) \le C + C_{2,\nu}.
\]
We now estimate the second integral by using \eqref{eq: part of partial K}:
\begin{align*}
\int_{0}^{1} \frac{1}{|1 - rz\overline{\lambda}|^2}\,d\nu(r) 
\leq 4 C_{2,\nu}\frac{1}{|1-z\overline{\lambda}|}.
\end{align*}
Combining these results, we find
\[
|\partial_{z}K_{\nu}(z,\lambda)| \le  (C + 5C_{2,\nu})\,\frac{1}{|1-z\overline{\lambda}|^{2}}.
\]
This shows that $K_{\nu}$ is a $1$-Bergman--Calderón--Zygmund kernel when $c_{\nu} = 2$.
\end{proof}

\section{The proof of Theorem \ref{pq-necessary}}
For the proof, we first collect several classical results on integral operators on the unit disk as follows.
\begin{lemma}[Proposition 5.2, \cite{T.Tao}]\label{Tao}
  Let $1\le q\le \infty$. Suppose $k:\mathbb{D}\times\mathbb{D}\to\mathbb{C}$ is
measurable and that $\|k(z,\cdot)\|_{L^{q}(\mathbb{D})}$ is uniformly bounded
for almost every $z\in\mathbb{D}$. Define the integral operator
\[
(Tf)(\lambda)=\int_{\mathbb{D}} k(z,\lambda)\,f(z)\,dA(z), \quad \lambda\in\mathbb{D}.
\]
Then $T:L^{1}(\mathbb{D})\to L^{q}(\mathbb{D})$ is bounded and
\[
\|T\|_{L^{1}(\mathbb{D})\to L^{q}(\mathbb{D})}
= \sup_{z\in\mathbb{D}} \|k(z,\cdot)\|_{L^{q}(\mathbb{D})}.
\]
Conversely, if $T:L^{1}(\mathbb{D})\to L^{q}(\mathbb{D})$ is bounded, then
$\|k(z,\cdot)\|_{L^{q}(\mathbb{D})}$ is uniformly bounded for almost every
$z\in\mathbb{D}$.
\end{lemma}

\begin{lemma}[Proposition 6.1 \cite{T.Tao}]\label{Taoestimate}
  Let $1<r<\infty$, $1<p<q<\infty$ with
\[
  \frac1p+\frac1r=\frac1q+1.
\]
Suppose that $k:\mathbb D\times\mathbb D\to\mathbb C$ is measurable such that
\[
  \|k(\cdot,\lambda)\|_{L^{r,\infty}(\mathbb D)} \le C
  \quad\text{for almost every } \lambda\in\mathbb D
\]
and
\[
  \|k(z,\cdot)\|_{L^{r,\infty}(\mathbb D)} \le C
  \quad\text{for almost every } z\in\mathbb D
\]
for some $C>0$. Then the operator $T$ defined as
\[
  Tf(\lambda)=\int_{\mathbb D} k(z,\lambda)f(z)\,dA(z)
\]
is bounded from $L^p(\mathbb D)$ to $L^q(\mathbb D)$.
\end{lemma}

\subsection{Proof of sufficiency}\label{S:suf1}
We prove the sufficiency part of Theorem \ref{pq-necessary}.

\subsubsection*{\bf The case $c_{\nu} = 1$.} 
Since $c_{\nu} = 1$, we have $c_{\nu}' = \infty.$
Recall Figure \ref{theorem 1 graph}. It suffices to prove that if $(p,q)$ satisfies one of the following conditions:
\begin{itemize}
	\item [(b)] $1 < p <\infty$, $1/q > 1/p$;
	\item [(c)] $p = \infty$, $1\leq q < \infty$,
\end{itemize}
then $T_{\nu} : L^p(\D) \to L^q(\D)$ is bounded.
Recall the universal estimates \eqref{universal_estimates}, 
\[
|K_{\nu}(z,\lambda)| \le \frac{2\nu([0,1])}{|1 - z\overline{\lambda}|^2} \quad z, \lambda \in \mathbb{D}.
\]
Define the positive operator $P_+$ by
\[
P_+ f(z)=\int_{\D}\frac{f(\lambda)}{|1-z\overline{\lambda}|^{2}}\,dA(\lambda),\quad f\in L^1(\mathbb{D}).
\]
Then, $P_+ : L^p(\mathbb{D}) \to L^q(\mathbb{D})$ is bounded implies the boundedness of $T_{\nu}: L^p(\mathbb{D}) \to L^q(\mathbb{D})$ for some $(p,q) \in [1,\infty]\times [1,\infty]$. 

\noindent\textbf{The case $1 < p <\infty$, $1/q > 1/p$.}
By Theorem 3.11 in \cite{Zhu14}, we have $P_+ : L^p(\mathbb{D}) \to L^p(\mathbb{D})$ is bounded for all $1 < p < \infty$, then $T_{\nu}: L^p(\mathbb{D}) \to L^p(\mathbb{D})$ is bounded for all $1 < p < \infty$. And for all $1 \le q < \infty$ such that 
\[
\frac{1}{q} > \frac{1}{p},
\]
we have for all $f \in L^p(\mathbb{D})$, 
\[\|T_{\nu} f\|_{L^q(\mathbb{D})} \le \|T_{\nu} f\|_{L^p(\mathbb{D})} \le \|T_{\nu}\|_{L^p(\mathbb{D}) \to L^p(\mathbb{D})} \|f\|_{L^p(\mathbb{D})}.\]
Hence, $T_{\nu}: L^p(\mathbb{D}) \to L^q(\mathbb{D})$ is bounded.

\noindent\textbf{The case $p = \infty$, $1\leq q < \infty$.} Choose $q_0$ such that $1\leq q <q_0< \infty$. Hence
\[
\|T_{\nu}f\|_{L^q(\mathbb{D})} \le \|T_{\nu}f\|_{L^{q_{0}}(\mathbb{D})}\le  \|T_{\nu}\|_{L^{q_{0}}(\mathbb{D})\to L^{q_{0}}(\mathbb{D})}  \|f\|_{L^{q_{0}}(\mathbb{D})}\le\|T_{\nu}\|_{L^{q_{0}}(\mathbb{D})\to L^{q_{0}}(\mathbb{D})} \|f\|_{L^{\infty}(\mathbb{D})}.
\]
Hence, $T_{\nu}: L^{\infty}(\mathbb{D}) \to L^q(\mathbb{D})$ is bounded for all $1 \le q < \infty.$
\medskip

\subsubsection*{\bf The case $1 < c_{\nu} \le 2$.} 
Let $\nu$ be a finite positive Borel measure on $[0,1]$. Recall the quantity $c_\nu$ defined by \eqref{critical-index}  
\begin{align*}
  c_\nu 
  &= \sup\left\{1 \le c < 2: \int_{0}^{1} \frac{1}{(1 - r^2)^{2/c'}} \,d\nu(r) < \infty\right\}.
\end{align*}
Then, $1 < c_\nu \le 2$ implies that $\nu(\{1\}) = 0$.

\noindent\textbf{The case $p=1,1\le q<c_\nu$.} 
We shall show that $T_{\nu} : L^p(\D) \to L^q(\D)$ is bounded.
Choose $q_0$ such that {$1\le q<q_0<c_\nu\le 2$}.
Hence, by Corollary \ref{Kp-norm}, we have
\[
\sup_{z\in\mathbb{D}}\|K_{\nu}(z,\cdot)\|_{L^{q_0}(\mathbb{D})}\leq \frac{2}{2 - q_0} \int_0^1 (1 - r)^{2/q_0 - 2} d\nu(r)<\infty.
\]
By Lemma \ref{Tao}, we immediately obtain that $T_{\nu} : L^1(\D) \to L^{q_0}(\D)$ is bounded.
Thus
\[\|T_\nu f\|_{L^q(\mathbb{D})}\leq \|T_\nu f\|_{L^{q_0}(\mathbb{D})}\leq  \|T_\nu \|_{L^1(\D) \to L^{q_0}(\D)}
\|f\|_{L^1(\D)},\]
then $T_\nu:L^1(\mathbb{D})\to L^q(\mathbb{D})$ is bounded for each $1\le q<c_\nu$.

\noindent\textbf{The case $1<p<c_\nu',1/q > 1/p + 1/c_{\nu} - 1.$}\label{S:1cv}
Since \[\frac{1}{q} > \frac{1}{p} + \frac{1}{c_{\nu}} - 1,\] then there exists $r$ such that $1<r < c_{\nu}<2$ such that \[\frac{1}{p} + \frac{1}{r} = 1 + \frac{1}{q}.\] 
By the upper bound in Corollary \ref{Kp-norm} and the definition of $c_\nu$, we have
\[
	\sup_{z\in \mathbb{D}} \|K_{\nu}(z,\cdot)\|_{L^r(\mathbb{D})} \le \frac{2}{2 - r} \int_0^1 (1 - t)^{2/r - 2} d\nu(t)<\infty.
\]
Since $\|K_{\nu}(z,\cdot)\|_{L^{r,\infty}(\mathbb D)}\leq \|K_{\nu}(z,\cdot)\|_{L^r(\mathbb{D})}$ and $K_{\nu}(z,\lambda)=\overline{K_{\nu}(\lambda,z)}$, by Lemma \ref{Taoestimate}, we see that $T_{\nu} : L^p(\D) \to L^q(\D)$ is bounded.

\noindent\textbf{The case $p =c_{\nu}',1\le q< \infty.$} The arguments in the above case can also apply the current situation. 

\noindent\textbf{The case $p > c_{\nu}',1\le q\le \infty.$}\label{S:pbig1} 
Recall that the dual operator of $T_\nu$ takes the form
\[
T_{\nu}^* g(z) = \int_{\mathbb{D}} K_{\nu}(\lambda,z)g(\lambda)\,dA(\lambda).
\]
Since $K_{\nu}(z,\lambda)=\overline{K_{\nu}(\lambda,z)}$, by Lemma \ref{Tao}, the boundedness of $T_{\nu}: L^1(\mathbb{D}) \to L^q(\mathbb{D})$ is equivalent to that of $T_{\nu}^*: L^1(\mathbb{D}) \to L^q(\mathbb{D})$, and hence to that of $T_{\nu}: L^{q'}(\mathbb{D}) \to L^{\infty}(\mathbb{D})$.
By the conclusion in the case $p=1$, we thus obtain that for $p > c_{\nu}'$ and $1 \le q \le \infty$, the operator $T_{\nu}: L^p(\mathbb{D}) \to L^{q}(\mathbb{D})$ is bounded.

\subsection{Proof of necessity}\label{S:necethm1}
Let $\nu$ be a positive finite measure on $[0,1]$. 
It suffices to prove that if $(p,q)$ satisfies one of the following conditions:
\begin{itemize}
	\item [(a)] $p=1$, $c_\nu<q\le \infty$;
	\item [(b)] $1<p<c_\nu'$, $q=\infty$;
	\item [(c)] $1<p<c_\nu'$, $1/q < 1/p + 1/c_{\nu} - 1$,
\end{itemize}
then $T_{\nu} : L^p(\D) \to L^q(\D)$ is unbounded.

\subsubsection*{\bf The case $p=1$, $c_{\nu} < q \le \infty$}
We proceed by contradiction. Suppose there exists $q > c_{\nu}$ such that $T_{\nu} : L^1(\D) \to L^q(\D)$ is bounded,
where $c_\nu$ is defined by \eqref{critical-index}
\[
c_\nu = \sup\left\{1 \le c < 2: \int_{0}^{1} (1 - r)^{2/c - 2} \,d\nu(r) < \infty\right\}.
\]
If $c_\nu=2$, then $\nu(\{1\}) = 0$. Hence, by Proposition~\ref{pnorm of K}, for each $q > 2 = c_{\nu}$, we have 
\[
\|K_{\nu}(z,\cdot)\|_{L^q(\D)}
\ge (1-|z|^2)^{2/q-1}\int_{0}^{1}\frac{1}{1-r|z|^2}\,d\nu(r)
\ge \nu([0,1]) (1-|z|^2)^{2/q-1}.
\]
Since $2/q-1<0$, letting $|z|\to 1^{-}$ gives
$\sup_{z\in\D}\|K_{\nu}(z,\cdot)\|_{L^q(\D)}=\infty$,
which contradicts the boundedness of $T_\nu:L^1(\D)\to L^q(\D)$.

If $1 \le c_{\nu} < 2$ and $\nu(\{1\}) = 0$, it suffices to prove that 
$T_\nu:L^1(\D)\to L^q(\D)$ is unbounded for each $q$ with $c_\nu<q<2$.
For simplicity, we still denote this exponent by $q$.
Consider the dyadic intervals $$I_j = [1 - 2^{-j}, 1 - 2^{-j-1})$$ for $j = 0, 1, 2, \cdots$. We first show there exists a constant $C > 0$ (independent of $j$) such that
\begin{equation}\label{nuestimate}
  \nu(I_j) \le C \cdot 2^{\left(2/q - 2\right)j}.
\end{equation}
To verify this, fix $z_j = (1 - 2^{-j})^{1/2}$. From the lower bound in Proposition \ref{pnorm of K}, we have
\[
\|K_{\nu}(z_j,\cdot)\|_{L^q(\D)} \geq (1 - |z_j|^2)^{2/q - 1} \int_0^1 \frac{d\nu(r)}{1 - r|z_j|^2}.
\] 
Note that $$|z_j|^2 = 1 - 2^{-j},$$ so we have
\begin{align*}
  \int_0^1 \frac{d\nu(r)}{1 - r|z_j|^2} 
  &\ge \int_{I_j }\frac{d\nu(r)}{1 - r|z_j|^2} = \int_{I_j } \frac{d\nu(r)}{1 - r\left(1 - 2^{-j}\right)} \\
  &\ge \int_{I_j} \frac{d\nu(r)}{1 - r + 2^{-j}}  \ge \frac{\nu(I_j)}{2^{-j} + 2^{-j}} = \frac{1}{2}\cdot 2^{j} \nu(I_j).
\end{align*}
Since $1 - |z_j|^2 = 1 - (1 - 2^{-j}) = 2^{-j}$, we obtain
\[
\|K_{\nu}(z_j,\cdot)\|_{L^q(\D)} \ge \frac{1}{2} \cdot (1 - |z_j|^2)^{2/q - 1} \cdot 2^{j} \nu(I_j) \ge \frac{1}{2}\cdot 2^{\left(2 - 2/q\right)j} \nu(I_j).
\]
Since $T_\nu : L^1(\D) \to L^q(\D)$ is bounded, Lemma \ref{Tao} implies $\|K_{\nu}(z_j,\cdot)\|_{L^q(\D)}$ is uniformly bounded for all $j$. Thus, there exists a constant $C > 0$ (independent of $j$) such that
\[
\nu(I_j) \le C \cdot 2^{\left(2/q - 2\right)j},
\]
which confirms \eqref{nuestimate}.

Next, fix $\varepsilon > 0$ sufficiently small such that \[q_{\varepsilon} = c_{\nu} + \varepsilon < q.\] Since $q_{\varepsilon} < q < 2$ and $T_\nu : L^1(\D) \to L^q(\D)$ is bounded, we compute
\begin{align*}
\int_0^1 (1 - r)^{2/q_{\varepsilon} - 2} d\nu(r)
&= \sum_{j=0}^{\infty} \int_{I_j} (1 - r)^{2/q_{\varepsilon} - 2} d\nu(r)  \le 2^{2 - 2/q_{\varepsilon}} \sum_{j=0}^{\infty} (2^{-j})^{2/q_{\varepsilon} - 2} \nu(I_j) \\
&\le  2^{2 - 2/q_{\varepsilon}}C \sum_{j=0}^{\infty} 2^{-j \cdot 2\left(1/q_{\varepsilon} - 1/q\right)}.
\end{align*}
The series converges because $1/q_{\varepsilon} - 1/q > 0$ (since $q_{\varepsilon} < q$), so
\[
\int_0^1 (1 - r)^{2/q_{\varepsilon} - 2} d\nu(r) < \infty.
\]
This contradicts the definition of $c_{\nu}$, as $q_{\varepsilon} > c_{\nu}$ (by construction) but the integral above is finite. 
Therefore, our initial assumption is false, and $T_{\nu}:L^{1}(\D)\to L^{q}(\D)$ is unbounded for every $q$ with $c_{\nu}<q<2$.

If $\nu(\{1\}) \ne 0$, by Proposition \ref{K estimates with nu1 neq 0}, we have
\[
|K_{\nu}(z,\lambda)| \ge \frac{\nu(\{1\})}{2} \frac{1}{|1 - z\overline{\lambda}|^2},
\]
which implies that for any $q > 1 = c_{\nu}$, by \cite[Lemma 3.10]{Zhu14}, 
\[
\sup_{z \in \mathbb{D}} \|K_{\nu}(z,\cdot)\|_{L^q(\mathbb{D})}^q \ge \frac{\nu(\{1\})^q}{2^q} \sup_{z \in \mathbb{D}} \int_{\mathbb{D}} \frac{1}{|1 - z\overline{\lambda}|^{2q}}dA(\lambda) = \infty.
\]
Then, $T_{\nu}: L^1(\mathbb{D}) \to L^q(\mathbb{D})$ is unbounded for all $1 < q < \infty$.

Only the case $p=1,q=\infty$ remains to be discussed.
If $T_{\nu}:L^{1}(\D)\to L^{\infty}(\D)$ were bounded, we would obtain $T_{\nu}:L^{1}(\D)\to L^{q}(\D)$ is  bounded for some $q\in(c_\nu,2)$, which contradicts the previous conclusion. Thus $T_{\nu}:L^{1}(\D)\to L^{q}(\D)$ is unbounded for all $q$ with $c_{\nu}<q\le \infty$. 
This completes the proof.

\subsubsection*{\bf The case $1 < p < c_\nu'$, $q = \infty$}
Since $K_\nu(z,w)=\overline{K_\nu(w,z)}$, we have $T_\nu=T_\nu^*$. By duality and the same arguments as Subsection \ref{S:pbig1}, we have if $1<p<c_\nu'$, $q=\infty$, then  $T_{\nu} : L^p(\mathbb{D}) \to L^{\infty}(\mathbb{D})$ is unbounded.

\subsubsection*{\bf The case $1<p<c_\nu'$, $1/q < 1/p + 1/c_{\nu} - 1$}\label{constract-cm}
In this setting, the argument splits into three cases according to the value of $c_\nu$ as follows:
\begin{itemize}
	\item [(a)] $c_\nu=1$;
	\item [(b)] $c_\nu=2$;
	\item [(c)] $1<c_\nu<2$.
\end{itemize}
When $c_\nu=1$ or $c_\nu=2$, we instead construct localized test functions supported in boundary boxes $E_t$ as in Lemma \ref{l-estimate of real part of K} and combine them with lower bounds for the real part of the Shimorin-type kernel to obtain the desired contradiction.
When $1<c_\nu<2$, we employ a multiplier method and a coefficient criterion for Bergman $L^p$-regularity of analytic functions.

\noindent\textbf{The case $c_\nu=1$.}
We aim to show that $T_{\nu} : L^p(\mathbb{D}) \to L^q(\mathbb{D})$ is unbounded with 
\[ 1<p<\infty\quad \text{and} \quad\frac{1}{q} < \frac{1}{p}.\] 
Recall the definition of $c_\nu$ by \eqref{critical-index}, we have
\begin{equation}\label{def_c_nu}
\int_{0}^{1}(1 - r)^{2/c - 2} d\nu(r)= \infty, \quad c > 1.
\end{equation}

\begin{lemma}\label{l-estimate of real part of K}
Let $t\in(0,1)$ and define the set
$$
E_t := \left \{\rho e^{i\theta} \in \mathbb{D}: 1 - t \le \rho \le 1 - \frac{t}{2},\ |\theta| \le \frac{t}{20} \right \}.
$$
There exist absolute constants $c_0 > 0$ and $t_0 \in (0,1)$(we can take $c_0 = \sqrt{3}/36$ and $t_0 =1/2$) such that for all $0 < t < t_0$ and all $z, \lambda \in E_t$, the following assertions hold:
\begin{itemize}
    \item[(1)] For every $r \in [0,1]$,
    \[
    \operatorname{Re}\left(\frac{1}{(1 - z \overline{\lambda})(1 - r z\overline{\lambda})} \right) \ge \frac{c_0}{t\left((1 - r) + t\right)}.
    \]

    \item[(2)] In particular, for every $ r \in [1 - t,1]$,
    \[
    \operatorname{Re}\left(\frac{1}{(1 - z\overline{\lambda})(1 - r z\overline{\lambda})} \right) \ge \frac{c_0}{2t^{2}}.
    \]
\end{itemize}
\end{lemma}

\begin{proof}
For $z, \lambda \in E_t$, denote 
\[
\vartheta =\arg\frac{1}{(1 - z\overline{\lambda})(1 - rz\overline{\lambda})},
\]
we have
\[
\operatorname{Re}\left( \frac{1}{(1 - z\overline{\lambda})(1 - rz\overline{\lambda})} \right) =  \frac{\cos\vartheta}{|1 - z\overline{\lambda}||1 - rz\overline{\lambda}|}.
\]
By the triangle inequality,
\begin{equation}\label{1-zl}
\begin{split}
|1 - z\overline{\lambda}| &\le \left|1 - |z| |\lambda|\right| + |z| |\lambda| \left|1 - \frac{z\overline{\lambda}}{|z||\lambda|}\right| \\
&\le 1 - \left(1 - t\right)^2 + \left|\arg z\overline{\lambda}\right| \\
&\le 2t - t^2 + \frac{t}{10} \le 3t.
\end{split}
\end{equation}
And for $0\leq r\leq1$,
\begin{equation}\label{1-rzl}
\begin{split}
|1 - rz\overline{\lambda}| 
&= |1 - r + r(1 - z\overline{\lambda})| \le (1 - r) + |1 - z\overline{\lambda}|\\
&\le (1 - r) + 3t \le 3\left((1 - r) + t\right).
\end{split}
\end{equation}
Take $t_0 = 1/2$. 
For $0 < t < t_0$, $0\leq r\leq1$ and $z, \lambda \in E_t$ we have
\[
\operatorname{Re}\left( \frac{1}{(1 - z\overline{\lambda})(1 - rz\overline{\lambda})} \right) \ge \frac{\cos\vartheta}{|1 - z\overline{\lambda}||1 - rz\overline{\lambda}|} \ge \frac{\cos\vartheta}{9} \frac{1}{t\left((1 - r) + t\right)}.
\]
In particular, if $r \in [1 - t, 1]$, then
\[
\operatorname{Re}\left( \frac{1}{(1 - z\overline{\lambda})(1 - rz\overline{\lambda})} \right) \ge \frac{\cos\vartheta}{9} \frac{1}{t\left((1 - r) + t\right)} \ge \frac{\cos\vartheta}{18}\frac{1}{t^2}.
\]
It remains to prove $\vartheta$ has an upper bound.
For each $z, \lambda \in E_t$,
\[
\operatorname{Re}(1 - z\overline{\lambda}) = 1 - \operatorname{Re}(z\overline{\lambda}) \ge 1 - |z| |\lambda| \ge 1 - \left(1 - \frac{t}{2}\right)^2 \ge \frac{t}{2}.
\] 
Similarly, $\operatorname{Re}(1 - rz\overline{\lambda}) \ge t/2$.
Moreover,
\[
|\operatorname{Im}(1 - z\overline{\lambda})| = |\operatorname{Im}(z\overline{\lambda})| = |z||\lambda| \sin\left(\arg z\overline{\lambda}\right) \le \frac{t}{10},
\]
and similarly $|\operatorname{Im}(1 - rz\overline{\lambda})| \le t/10$.
Hence,
\begin{equation}\label{arg}
\left| \arg(1 - z\overline{\lambda}) \right| \le \arctan\left( \frac{|\operatorname{Im}(1 - z\overline{\lambda})|}{\operatorname{Re}(1-z\overline{\lambda})} \right) \le \arctan\left( \frac{t / 10}{t / 2} \right) = \arctan\left( \frac{1}{5} \right) < \frac{\pi}{12}.
\end{equation}
Also, we have $\left| \arg(1 - rz\overline{\lambda}) \right| < \pi/12$, and therefore
\[
\left| \arg\left[(1 - z\overline{\lambda})(1 - rz\overline{\lambda})\right] \right| \le \left| \arg(1 - z\overline{\lambda}) \right| + \left| \arg(1 - rz\overline{\lambda}) \right| < \frac{\pi}{6}.
\]
Thus,
\[
\left| \arg\frac{1}{(1 - z\overline{\lambda})(1 - rz\overline{\lambda})} \right| = \left| \arg\left[(1 - z\overline{\lambda})(1 - rz\overline{\lambda})\right] \right| < \frac{\pi}{6},\quad \cos \vartheta > \cos{\frac{\pi}{6}}.
\]
Taking $c_0=\sqrt{3}/36$ completes the proof.
\end{proof}
\begin{lemma}\label{lem:new19}
Let $E_t$ be defined in Lemma \ref{l-estimate of real part of K}. There exist absolute constants $\widetilde{c_0} > 0$ and $t_0 \in (0,1)$ (we can take $\widetilde{c_0} = 1/108$ and $t_0 =1/2$) such that for all $0 < t < t_0$ and all $z, \lambda \in E_t$, the following assertions hold:
\begin{itemize}
    \item[(1)] For every $r \in [0,1]$,
    \[
    \operatorname{Re}\left(\frac{\overline{\lambda}}{(1 - z \overline{\lambda})^2(1 - r z\overline{\lambda})} \right) \ge \frac{\widetilde{c_0}}{t^2\left((1 - r) + t\right)}, 
    \]
    and
    \[
    \operatorname{Re} \left(\frac{\overline{\lambda}}{(1 - z\overline{\lambda})(1 - r z\overline{\lambda})^2} \right) \ge \frac{\widetilde{c_0}}{t\left((1 - r) + t\right)^2}.
    \]

    \item[(2)] In particular, for every $r \in [1 - t, 1]$,
    \[
    \operatorname{Re} \left(\frac{\overline{\lambda}}{(1 - z\overline{\lambda})^2(1 - r z\overline{\lambda})} \right) \ge \frac{\widetilde{c_0}}{2t^3}, \quad 
    \operatorname{Re} \left(\frac{\overline{\lambda}}{(1 - z\overline{\lambda})(1 - r z\overline{\lambda})^2} \right) \ge \frac{\widetilde{c_0}}{4t^3}.
    \]
\end{itemize}
\end{lemma}

\begin{proof}
Let $z, \lambda \in E_t$. By \eqref{arg}, we have $|\arg(1 - z\overline{\lambda})| < \pi/12$ and $|\arg(1 - rz\overline{\lambda})| < \pi/12$. Furthermore, since $\lambda \in E_t$, we have $|\arg \overline{\lambda}| = |\arg \lambda| \le t/20$. Thus, the total argument of the first expression satisfies
\[
\left| \arg\left( \frac{\overline{\lambda}}{(1 - z \overline{\lambda})^2(1 - r z\overline{\lambda})} \right) \right| \le \frac{t}{20} + \frac{\pi}{4}.
\]
Take $t_0=1/2$, since $0<t<t_0$, this total argument is strictly less than $\pi/3$. 
Consequently, 
\[
\operatorname{Re}\left(\frac{\overline{\lambda}}{(1 - z \overline{\lambda})^2(1 - r z\overline{\lambda})} \right) \ge \cos\left(\frac{\pi}{3}\right) \frac{|\lambda|}{|1 - z\overline{\lambda}|^2 |1 - rz\overline{\lambda}|} \ge \frac{1}{2} \frac{1-t}{|1 - z\overline{\lambda}|^2 |1 - rz\overline{\lambda}|}.
\]
Note that $|\lambda|\geq 1-t\geq \frac{1}{2}$ for each $\lambda\in E_t$.
By \eqref{1-zl} and \eqref{1-rzl}, we have
\[
\operatorname{Re}\left(\frac{\overline{\lambda}}{(1 - z \overline{\lambda})^2(1 - r z\overline{\lambda})} \right) \ge \frac{1}{4} \frac{1}{(3t)^2 \cdot 3((1-r) + t)} = \frac{1}{108} \frac{1}{t^2((1-r) + t)}.
\]
Take $\widetilde{c_0}=1/108.$
The second inequality in assertion (1) follows by an identical argument. 

Finally, for $r \in [1 - t, 1]$, we have $(1 - r) + t \le 2t$. Substituting this into the results of assertion (1) yields the lower bounds in assertion (2) with the same constant $\widetilde{c_0}$.
\end{proof}

%For the remainder of the proof, we require the following estimate whose proof is given in Subsection \ref{prooflemma-kernel-low}.
Let $c_0$ and $t_0$ be the constants in Lemma \ref{l-estimate of real part of K}. For any $0<t<t_0$, pick $z_t \in E_t$ and define a function
\begin{equation}\label{f_{z_t}}
f_{z_t}(\lambda) = \frac{\overline{K_{\nu}(z_t,\lambda)}}{|K_{\nu}(z_t,\lambda)|} \mathbf{1}_{E_t}(\lambda), \quad \lambda \in \mathbb{D}.
\end{equation}
Then, we have $\|f_{z_t}\|_{L^p(\D)} = |E_{t}|^{\frac{1}{p}}$.
It suffices to show for $1<p<\infty, 1/q < 1/p$, we have
\[
	\sup_{0 < t < t_0}\frac{\|T_{\nu}f_{z_t}\|_{L^q(\mathbb{D})}}{\|f_{z_t}\|_{L^p(\mathbb{D})}}=\infty.
\]
Since there exist absolute constants $c_1$ and $c_2$ such that
\begin{equation}\label{|E_t|}
  c_1 t^2 \le |E_t| \le c_2 t^2,
\end{equation}
we have
\begin{equation}\label{||f||}
	c_1^{1/p} t^{2/p} \le \|f_{z_t}\|_{L^p(\mathbb{D})} \le c_2^{1/p} t^{2/p}.
\end{equation}
Moreover,  
\begin{align*}
|T_{\nu}f_{z_t}(z_t)| 
&= \int_{\mathbb{D}} \frac{\overline{K_{\nu}(z_t,\lambda)}}{|K_{\nu}(z_t,\lambda)|} \mathbf{1}_{E_t}(\lambda) K_{\nu}(z_t,\lambda)\,dA(\lambda) \\
&= \int_{E_t} \left |\frac{1}{1 - z_t\overline{\lambda}} \int_0^1 \frac{1}{1 - r z_t\overline{\lambda}} \,d\nu(r)\right | \,dA(\lambda).
\end{align*}
By Lemma \ref{l-estimate of real part of K}, for all $z, \lambda\in E_t$ and $r\in[0,1]$,
the integrand defining $T_\nu f_{z_t}(z_t)$ has nonnegative real part. Then,
\begin{equation}\label{|Tf|1}
\begin{aligned}
|T_{\nu}f_{z_t}(z_t)|&\ge \int_{E_t} \left | \operatorname{Re}  \int_{0 }^1 \frac{1}{(1 - z_t\overline{\lambda})(1 - r z_t\overline{\lambda})} \,d\nu(r) \right | \,dA(\lambda)  \\
&\ge \int_{E_t} \int_{1 - t}^1 \operatorname{Re} \frac{1}{(1 - z_t\overline{\lambda})(1 - r z_t\overline{\lambda})} \,d\nu(r) \,dA(\lambda)  \\
&\ge \int_{1 - t}^{1} \frac{c_0}{2t^{2}} \,d\nu(r) |E_t| \ge \frac{1}{2}c_0 c_1 \nu([1-t,1]).
\end{aligned}
\end{equation}
Since $|z_t| < 1 - t/2$, the ball
$$B\left(z_t, \frac{t}{4}\right) = \left \{\lambda \in \mathbb{C}: |\lambda - z_t| < \frac{t}{4}\right \}$$
is contained in $\mathbb{D}$. By the subharmonicity of $|T_{\nu}f_{z_t}|^q$ at $z_t$, we obtain
\begin{equation}\label{below_|Tf|}
|T_{\nu}f_{z_t}(z_t)|^q 
\le \frac{1}{\left|B\left(z_t, t/4\right)\right|} \int_{B\left(z_t, t/4\right)} |T_{\nu}f_{z_t}(\lambda)|^q \,dA(\lambda)
\le  \frac{16}{t^2} \int_{\mathbb{D}} |T_{\nu}f_{z_t}(\lambda)|^q \,dA(\lambda).
\end{equation}
Raising both sides to the power $1/q$ and combining with \eqref{|Tf|1}, we deduce
\[
\|T_{\nu}f_{z_t}\|_{L^q(\mathbb{D})} \ge 16^{-1/q} t^{2/q} |T_{\nu}f_{z_t}(z_t)| \ge \frac{1}{2} 16^{-1/q}c_0 c_1 t^{2/q} \nu([1-t,1]).
\]
Dividing both sides by $\|f_{z_t}\|_{L^p(\mathbb{D})}$ and applying \eqref{||f||}, we have
\[
\frac{\|T_{\nu}f_{z_t}\|_{L^q(\mathbb{D})}}{\|f_{z_t}\|_{L^p(\mathbb{D})}} \ge C t^{2/q - 2/p}\nu([1 - t,1]),
\]
where $C = \frac{1}{2} 16^{-1/q} c_0 c_1c_2^{-1/p}$.
Hence, it suffices to demonstrate that
\[
\sup_{0 < t \le 1} t^{2/q - 2/p} \nu([1 - t, 1]) = \infty.
\]
If $\nu(\{1\}) \ne 0$, then 
\[
\sup_{0 < t \le 1} t^{2/q - 2/p} \nu([1 - t, 1]) \ge \sup_{0 < t \le 1} t^{2/q - 2/p} \nu(\{1\}) = \infty.
\]
If $\nu(\{1\}) = 0$, we prove it by contradiction, assume this supremum were finite, then there would exist a constant $\widetilde{C} > 0$ such that
\[
\nu([1 - t, 1]) \le \widetilde{C} t^{2/p - 2/q}, \quad 0 < t \le 1.
\]
We can take $t_j = 2^{-j}$, and $I_j = [1 - 2^{-j}, 1 - 2^{-j-1})$ for $j = 0, 1, 2, \cdots$. Then we have $$\nu(I_j) \le \nu([1 - t_j, 1]) \le \widetilde{C} 2^{-j(2/p - 2/q)}$$ for all $j = 0, 1, 2, \cdots$. Therefore, we can take $s = \min \left\{1/2, 1/p - 1/q\right\} > 0$. Then
\begin{align*}
\int_0^1 (1 - r)^{-s} d\nu(r) 
&= \sum_{j=0}^{\infty} \int_{I_j} (1 - r)^{-s} \,d\nu(r)  \le \sum_{j=0}^{\infty} (2^{-j})^{-s} \nu(I_j) \\
&\le \widetilde{C} \sum_{j=0}^{\infty} (2^{-j})^{-(1/p - 1/q)} \cdot 2^{-j(2/p - 2/q)} = \widetilde{C} \sum_{j=0}^{\infty} 2^{-j(1/p - 1/q)} < \infty.
\end{align*}
Then we can take $c = 2/(2 - s) <2$, such that 
\[
\int_{0}^{1}(1 - r)^{2/c - 2} d\nu(r) = \int_0^1 (1 - r)^{-s} d\nu(r) < \infty.
\]
This contradicts the definition \eqref{def_c_nu} of $c_{\nu}$. Hence, we conclude that $T_{\nu} : L^p(\mathbb{D}) \to L^q(\mathbb{D})$ is unbounded.

\noindent\textbf{The case $c_\nu=2.$} 
We aim to show that $T_{\nu} : L^p(\mathbb{D}) \to L^q(\mathbb{D})$ is unbounded with
\[ 1<p<2\quad \text{and} \quad \frac{1}{q} < \frac{1}{p} - \frac{1}{2}.\] 

Let $c_0$ and $t_0$ be the constants in Lemma \ref{l-estimate of real part of K}. For any $0<t<t_0$, pick $z_t \in E_t$ and define a function $f_{z_t}$ as \eqref{f_{z_t}}.
It suffices to show for $1<p<\infty, 1/q < 1/p-1/2$, we have
\[
	\sup_{0 < t < t_0}\frac{\|T_{\nu}f_{z_t}\|_{L^q(\mathbb{D})}}{\|f_{z_t}\|_{L^p(\mathbb{D})}}=\infty.
\]
By Lemma \ref{l-estimate of real part of K}, for all $z, \lambda\in E_t$ and $r\in[0,1]$,
the integrand defining $T_\nu f_{z_t}(z_t)$ has nonnegative real part and combining \eqref{|Tf|1} Then,
\begin{equation}\label{|Tf|2}
\begin{aligned}
|T_{\nu}f_{z_t}(z_t)|
&\ge   \int_{E_t} \int_{0 }^1 \operatorname{Re} \frac{1}{(1 - z_t\overline{\lambda})(1 - r z_t\overline{\lambda})} \,d\nu(r) \,dA(\lambda) \\
&\ge \int_{0}^{1} \frac{c_0}{t\left((1 - r) + t\right)} \,d\nu(r) |E_t| \ge c_0 c_1 t \int_{0}^{1} \frac{1}{(1 - r) + t} \,d\nu(r).
\end{aligned}
\end{equation}
By \eqref{below_|Tf|}, we obtain
\[
|T_{\nu}f_{z_t}(z_t)|^q \le  \frac{16}{t^2} \int_{\mathbb{D}} |T_{\nu}f_{z_t}(\lambda)|^q \,dA(\lambda).
\]
Raising both sides to the power $1/q$ and combining with the earlier estimate for $|T_{\nu}f_{z_t}(z_t)|$, we deduce
\[
\|T_{\nu}f_{z_t}\|_{L^q(\mathbb{D})} \ge 16^{-1/q} t^{2/q} |T_{\nu}f_{z_t}(z_t)| \ge 16^{-1/q}c_0 c_1 t^{2/q + 1} \int_{0}^{1} \frac{1}{(1 - r) + t} \,d\nu(r).
\]
Dividing both sides by $\|f_{z_t}\|_{L^p(\mathbb{D})}$ and applying the norm bound for $f_{z_t}$, by \eqref{|Tf|2}, let $C = 16^{-1/q} c_0 c_1c_2^{-1/p}$. Then
\[
\frac{\|T_{\nu}f_{z_t}\|_{L^q(\mathbb{D})}}{\|f_{z_t}\|_{L^p(\mathbb{D})}} \ge C t^{2/q - 2/p + 1}\int_{0}^{1} \frac{1}{(1 - r) + t} \,d\nu(r).
\]
Note that by Monotone Convergence Theorem, we have
\[
	\lim\limits_{t\to 0}\int_{0}^{1} \frac{1}{(1 - r) + t} d\nu(r)=\int_{0}^{1} \frac{1}{(1 - r) }d\nu(r)\geq \nu([0,1])>0.
\]
Since
$2/q - 2/p + 1 < 0$, we find that
\[
\lim_{t\to 0}t^{2/q - 2/p + 1}\int_{0}^{1} \frac{1}{(1 - r) + t} \,d\nu(r) = \infty.
\]
It follows that
\[
\lim_{t\to 0}\frac{\|T_{\nu}f_{z_t}\|_{L^q(\mathbb{D})}}{\|f_{z_t}\|_{L^p(\mathbb{D})}} =\infty.
\]
We thus conclude that $T_{\nu} : L^p(\mathbb{D}) \to L^q(\mathbb{D})$ is unbounded.

\noindent\textbf{The case $1<c_\nu<2.$}
We aim to show that 
$T_\nu:L^p(\D)\to L^q(\D)$ is unbounded with
\[
	1<p<c_\nu'\quad \text{and} \quad \frac{1}{q}< \frac{1}{p}+\frac{1}{c_\nu}-1.
\]
Note that $T_\nu f$ is always holomorphic for $f\in L^1(\D)$.
Hence, it suffices to show that if $1 < p < c_{\nu}'$, and $1/q < 1/p + 1/c_{\nu} - 1$, 
there exists a function $f$ such that $f\in L^p_a(\D)$ but $T_\nu f\notin L^q_a(\D)$.

We record  \cite[Proposition 2.4]{BKV} as Lemma \ref{Lpestimate}, which will be used to construct analytic test functions with prescribed Taylor coefficients.

\begin{lemma}\label{Lpestimate}
  Suppose $a_n\ge 0$, $1<p<\infty$, and $C<\infty$, and that either
\begin{itemize}
  \item[(1)] $(a_n)$ is monotonic; or
  \item[(2)] for each $k$, $B_k = \{n: 2^{k} \le n < 2^{k+1}\}$, the subsequence $\left(a_n\right)_{n\in B_k}$ is monotonic and, for all $n,m\in B_k$, one has $a_m\le C\,a_n$.
\end{itemize}
Then $f(z) = \sum_{n=0}^{\infty} a_n z^n \in L_a^p(\mathbb{D})$ if and only if
\[
\sum_{n=0}^{\infty} (n+1)^{\,p-3}\,|a_n|^{p}<\infty .
\]
\end{lemma}

Let $s_0$ be defined as in \eqref{eq:s0_def}, then 
\[
0 < s_0 = \frac{2}{c_{\nu}'}<1.
\]
Since $1/q < 1/p + 1/c_{\nu} - 1$, we have
\[
\frac{2}{p}-1>\frac{2}{q}+\frac{2}{c_\nu'}-1=\frac{2}{q}+s_0-1.
\]
Hence, there is a constant $t$ such that
\begin{align}\label{findt}
\frac{2}{p} - 1 > t > \frac{2}{q} + s_0 - 1.
\end{align}
Take $\varepsilon>0$ small enough such that 
\begin{align}\label{findeps1}
	 t > \frac{2}{q} + s_0 - 1+2\varepsilon,
\end{align} 
and
\begin{align}\label{findeps2}
	\varepsilon< \frac{1 - s_0}{q + 1}<1 - s_0.
\end{align} 
By Corollary \ref{prop:sum_of_mn}, there is an increasing subsequence denoted by $\Lambda$ such that
\begin{align*}
\Lambda = \left\{ n \in\N: m_n \ge (n+1)^{-s_0 - \varepsilon}\right\}
\end{align*}
and
\begin{align}\label{infinitesum}
\sum_{n \in \Lambda} m_{n} (n + 1)^{s - 1} = \infty,
\end{align}
for $s_0<s < 1$.

We divide the argument into two cases: $2\leq p <c_\nu'$ and $1< p<2$.
\medskip

\noindent\textit{\textbf{The case $2\leq p <c_\nu'$.}}
In this case, $t < 2/p - 1 \le 0$. Let 
\begin{align}\label{testingft}
f_t(z) = \sum_{n = 0}^{\infty} (n+1)^t z^n.\end{align}
Since $t < 2 / p - 1$, we have $tp + p - 3 < -1$.
Then 
\[
\sum_{n=0}^{\infty} (n+1)^{p - 3} (n+ 1)^{tp} = \sum_{n=0}^{\infty} (n+1)^{tp + p - 3} < \infty.
\]
By the Lemma \ref{Lpestimate}, we have $f_t \in L_a^p(\mathbb{D})$.

Recall the definition of $m_n$ given in \eqref{momentsequence}:
\begin{align*}
m_n &= \frac{1}{n+1} \int_{0}^{1} \frac{1 - r^{n+1}}{1 - r} d\nu(r) = 
\int_{0}^{1} \frac{1}{n+1} \sum_{k=0}^n r^k \, d\nu(r).
\end{align*}
Since $r^{n+1} \leq r^k$ for all $0 \leq k \leq n$, it follows that
\[
r^{n+1} \leq \frac{1}{n+1} \sum_{k=0}^{n} r^k.
\]
Furthermore, we compute $m_{n+1}$ as follows:
\begin{align*}
m_{n+1} 
  &= \int_{0}^{1} \frac{1}{n+2} \sum_{k=0}^{n+1} r^k \, d\nu(r) = \int_{0}^{1} \frac{n+1}{n+2} \left( \frac{1}{n+1}\sum_{k=0}^{n} r^k + \frac{r^{n+1}}{n+1} \right) d\nu(r)\\
  &\leq \int_{0}^{1} \frac{n+1}{n+2} \left( \frac{1}{n+1}\sum_{k=0}^{n} r^k + \frac{1}{(n+1)^2}\sum_{k=0}^{n} r^k \right) d\nu(r) \\
  &= \int_{0}^{1} \frac{1}{n+1} \sum_{k=0}^{n} r^k \, d\nu(r) = m_n.
\end{align*}
We thus conclude that the sequence $\{m_n\}$ is non-increasing. Using the coefficient representation of $T_\nu$ (cf. \eqref{tcm}), we have 
\[
T_{\nu} f_t(z) = \sum_{n=0}^{\infty} m_n (n+1)^t z^n.
\]
Since $t < 0$ and $\{m_n\}$ is non-increasing and positive, it follows that the sequence $\{m_n (n+1)^t\}$ is non-increasing as well. We claim that
\[
T_{\nu} f_t \notin L_a^q(\mathbb{D}).
\]
To verify this claim, by Lemma \ref{Lpestimate}, it suffices to show
\[
\sum_{n=0}^{\infty} (n+1)^{q - 3} \left(m_n (n+1)^t\right)^q = \infty.
\]
Indeed, for all $n \in \Lambda$, we have the lower bound $m_n \ge (n+1)^{-s_0 - \varepsilon}$. Substituting this estimate into the series yields
\begin{equation}\label{TfLq1}
\begin{aligned}
  \sum_{n=0}^{\infty} (n+1)^{q - 3} \left(m_n (n+1)^t\right)^q
  &\ge \sum_{n \in \Lambda} m_n \cdot m_n^{q - 1} (n+1)^{tq + q - 3} \\
  &\ge \sum_{n \in \Lambda} m_n (n + 1)^{\left(-s_0 - \varepsilon\right)(q - 1)} (n+1)^{tq + q - 3}\\
  &= \sum_{n \in \Lambda} m_n (n+1)^{\left(-s_0 - \varepsilon\right)(q - 1) + tq + q - 3}.
\end{aligned}
\end{equation}
By \eqref{findeps1}, we have
\[ tq + q - 3>\left(\frac{2}{q} + s_0 + 2\varepsilon\right)q - 3.\]
Hence,
\begin{equation}\label{TfLq2}
\begin{aligned}
  \sum_{n=0}^{\infty} (n+1)^{q - 3} \left(m_n (n+1)^t\right)^q
  &\ge \sum_{n \in \Lambda} m_n (n+1)^{\left(-s_0 - \varepsilon\right)(q - 1) + \left(2/q + s_0 + 2\varepsilon\right)q - 3} \\
  &= \sum_{n \in \Lambda} m_n (n+1)^{s_0 + \varepsilon(q + 1) - 1}.
\end{aligned}
\end{equation}
By the choice of $\varepsilon$ (cf. \eqref{findeps2}), we have $1>s_0 + \varepsilon(q + 1) > s_0$. By \eqref{infinitesum}, we obtain
\[
\sum_{n \in \Lambda} m_n (n+1)^{s_0 + \varepsilon(q + 1) - 1} = \infty.
\]
Lemma \ref{Lpestimate} shows that $T_{\nu} f_t \notin L_a^q(\mathbb{D})$.
\medskip

\noindent\textit{\textbf{The case $1< p< 2$.}} 
In this case, $2/p - 1 > 0$. By \eqref{findt}, there exists a constant $t$ such that
\begin{align*}
\frac{2}{p} - 1 > t > \frac{2}{q} + \frac{2}{c_{\nu}'} - 1.
\end{align*}
If $t \le 0$, then the sequence $\{m_n(n+1)^t\}$ is non-increasing.
Hence, we reduce to the case $p \ge 2$.
If $t > 0$, however, the sequence $\{m_n(n+1)^t\}$ is not necessarily monotonic in general. 
We therefore introduce a suitable test function to proceed with our argument.

For each non-negative integer $k$ (i.e., $k = 0, 1, 2, \dots$), we define the block of non-negative integers
\[
B_k = \{n: 2^k \leq n < 2^{k+1}\},
\]
Let $l_k = 2^k$. 
For any $n \in B_k$, we have
\begin{equation}\label{block-m}
m_n = \frac{\int_0^1 \frac{1 - r^{n+1}}{1 - r} d\nu(r)}{n+1} \geq \frac{\int_0^1 \frac{1 - r^{l_k+1}}{1 - r} d\nu(r)}{n+1} = \frac{l_k+1}{n+1} m_{l_k} \geq \frac{1}{2}m_{l_k}.
\end{equation}
Define the coefficient sequence in each $B_k$ for $k = 0, 1, 2, \cdots$
\[
a_n = 
  \frac{m_{l_k}}{m_n}(n+1)^t,  \quad n \in B_k, \text{ and } k = 0, 1, 2,\cdots.
\]
We then introduce our modified test function as
\begin{align}\label{testinggt}
g_t(z) = \sum_{n=0}^{\infty}a_n z^n.
\end{align}
Note that,
\[
\sum_{n=0}^{\infty} |a_n|^p (n+1)^{p - 3} \le 2^p  \sum_{n=0}^{\infty} (n+1)^{tp} (n+1)^{p - 3} = 2^p  \sum_{n=0}^{\infty} (n+1)^{p(t+1) - 3}.
\]
Since $t < 2/p - 1$ which implies that $p(t+1) - 3 < -1$, we have
\[
	\sum_{n=0}^{\infty} |a_n|^p (n+1)^{p - 3} < \infty.
\]
Moreover, since $t > 0$, we have $a_n = \frac{m_{l_k}}{m_n}(n+1)^t$ is increasing in each $B_k$.
By \eqref{block-m}, for all $n, l \in B_{k}$, 
\[
\frac{a_n}{a_l} = \frac{(n+1)^t}{(l+1)^t} \frac{m_l}{m_n} \le 2^{t+1}.
\]
Thus, by Lemma \ref{Lpestimate}, $g_t\in L_a^p(\mathbb{D})$.

Now, we show that $T_{\nu}g_t \notin L_a^q(\mathbb{D})$. $T_{\nu}g_t(z) = \sum_{k = 0}^{\infty}\sum_{n \in B_k} m_{l_k} (n+1)^t z^n$. For each $k\in\N$ and $n\in B_k$, $\{m_{l_k} (n+1)^t\}$ is monotonic, and
\[
\frac{(n+1)^t m_{l_k}}{(l+1)^tm_{l_k}} = \frac{(n+1)^t }{(l+1)^t} \le 2^t, \quad n, l \in B_k.
\]
Then, by Lemma \ref{Lpestimate}, we only need to show that $\sum_{k = 0}^{\infty}\sum_{n \in B_k} (m_{l_k} (n+1)^t)^q (n + 1)^{q - 3} = \infty$. Since $m_n$ is decreasing, 
\begin{align*}
  \sum_{k = 0}^{\infty}\sum_{n \in B_k} (m_{l_k} (n+1)^t)^q (n + 1)^{q - 3}
  &\ge \sum_{k = 0}^{\infty}\sum_{n \in B_k} (m_{n} (n+1)^t)^q (n + 1)^{q - 3}\\
  &\ge \sum_{n = 0}^{\infty}(m_{n} (n+1)^t)^q (n + 1)^{q - 3}.
\end{align*}
The same argument as in \eqref{TfLq1} and \eqref{TfLq2} completes the proof.

\section{The proof of Theorem \ref{thm:new2}}\label{Thm:2-3}
For the proof of Theorem \ref{thm:new2}, in this section, we state and prove three results on the critical boundary.
For simplication of the notation,
for any real number $\alpha\geq 0$, a positive finite measure $\nu$ on $[0,1)$ is said to satisfy the $\alpha$-Carleson condition if
\begin{align*}
\sup_{0 < t \le1}\frac{\nu([1-t,1))}{t^\alpha} < \infty.
\end{align*}
In particular, the measure $\nu$ satisfies the $0$-Carleson condition means that it is a finite measure. A positive measure $\nu$ is integrable with respect to the hyperbolic arc length on $[0,1)$ if
\begin{align*}
\int_{[0,1)} \frac{1}{1-r^2} \, d\nu(r) < \infty.
\end{align*}

\begin{theorem}\label{weak bd on the line}
Let $\nu$ be a finite positive Borel measure on $[0,1]$. Let $c_\nu'$ be the dual index of $c_\nu$. Suppose that $1 < p < c_{\nu}'$, and consider $q$ satisfying the identity
\[
\frac{1}{q} = \frac{1}{p} + \frac{1}{c_{\nu}} - 1.
\]
Then the following hold:
\begin{itemize}
  \item [(a)] If $1 \le c_{\nu} < 2$, the following conditions are equivalent:
\begin{itemize}
  \item[(1)] $T_{\nu} : L^p(\mathbb{D}) \to L^{q,\infty}(\mathbb{D})$ is bounded;
  \item[(2)] $T_{\nu} : L^p(\mathbb{D}) \to L^q(\mathbb{D})$ is bounded;
  \item[(3)] The measure $\nu$ satisfies the $2/c_\nu'$-Carleson condition.
\end{itemize}
  \item [(b)] If $c_{\nu} = 2$ (whence $1 < p < 2$ and $q = \frac{2p}{2-p}$), the following conditions are equivalent:
\begin{itemize}
 \item[(1)] $T_{\nu} : L^p(\mathbb{D}) \to L^{q,\infty}(\mathbb{D})$ is bounded;
  \item[(2)] $T_{\nu} : L^p(\mathbb{D}) \to L^q(\mathbb{D})$ is bounded;
  \item[(3)] The measure $\nu$ is integrable with respect to the hyperbolic arc length on $[0,1)$.
\end{itemize}
\end{itemize}
\end{theorem}
\begin{theorem}\label{weak bd}
Let $\nu$ be a finite positive Borel measure on $[0,1]$. Let $c_\nu'$ be the dual index of $c_\nu$.  Then
  \begin{itemize}
    %\item[(a)] If $c_\nu=1$, $T_{\nu}: L^1(\mathbb{D}) \to L^{1,\infty}(\mathbb{D})$ is bounded 
    \item[(a)] If $1 \leq c_{\nu} < 2$, $T_{\nu}: L^1(\mathbb{D}) \to L^{c_{\nu},\infty}(\mathbb{D})$ is bounded if and only if the measure $\nu$ satisfies the $2/c_\nu'$-Carleson condition.
    \item[(b)] If $c_{\nu} = 2$, $T_{\nu}: L^1(\mathbb{D}) \to L^{2,\infty}(\mathbb{D})$ is bounded if and only if
the measure $\nu$ is integrable with respect to the hyperbolic arc length on $[0,1)$.
  \end{itemize}
\end{theorem}
As noted in Remark~\ref{BMO-bloch} in the introduction, we prove the following Bloch space version of Theorem~\ref{dual of weak 1 alpha}.

\begin{theorem}\label{dual of weak 1 alpha}
  Let $\nu$ be a finite positive Borel measure on $[0,1]$. Then
  \begin{itemize}
    \item[(a)] If $1 \le c_{\nu} < 2$, the operator $T_{\nu}: L^{c_{\nu}'}(\mathbb{D}) \to \mathcal{B}$ is bounded if and only if
the measure $\nu$ satisfies the $2/c_\nu'$-Carleson condition. 
    \item[(b)] If $c_{\nu} = 2$, the operator $T_{\nu}: L^2(\mathbb{D}) \to \mathcal{B}$ is bounded if and only if
the measure $\nu$ is integrable with respect to the hyperbolic arc length on $[0,1)$.
  \end{itemize}
\end{theorem}
\begin{proof}[Proof of Theorem \ref{thm:new2}]
It follows from Theorem \ref{weak bd on the line}, Theorem \ref{weak bd} and Theorem \ref{dual of weak 1 alpha}.
\end{proof}
Given a finite measure $\nu$ on $[0,1]$, under the $2/c_\nu'$-Carleson condition or the hyperbolic integrability condition appearing in Theorem \ref{weak bd on the line}, Theorem \ref{weak bd}, or Theorem \ref{dual of weak 1 alpha}, the operator $T_\nu$ is shown to be a $2/c_\nu$-Bergman–CZO on the unit disk by Lemma \ref{u-estimate of K}, Corollary \ref{cor:forcnu1} and Lemma \ref{u-estimate of K-2}. Therefore, applying Proposition \ref{alpha Bergman CZO} yields the implication $(3) \Rightarrow (2)$ in Theorem \ref{weak bd on the line}, as well as the sufficiency in Theorem \ref{weak bd} and Theorem \ref{dual of weak 1 alpha}.

\medskip
If $c_{\nu} = 1$, $2/c_{\nu}'$-Carleson condition is equivalent to $\nu$ is a finite measure. 
Starting from the next subsections, we complete the remaining parts of the above three theorems, which only concern the case \(1 < c_\nu \leq 2\). In particular, we have \(\nu(\{1\}) = 0\).

\subsection{Proof of Theorem \ref{weak bd on the line} }
The proof proceeds by considering two cases: $c_\nu\in(1,2)$ and $c_\nu=2$. Our strategy is 
\[
(1)\Rightarrow (3)\Rightarrow (2) \Rightarrow (1).
\]
Since the implication $(2)\Rightarrow (1)$ is trivial and $(3)\Rightarrow (2)$ has been proved, it suffices to establish the following implication:
\[
(1)\Rightarrow (3).
\]
Let $1 < p < c_{\nu}'$ with $1/q = 1/p + 1/c_{\nu} - 1$. Suppose that $T_{\nu} : L^p(\mathbb{D}) \to L^{q,\infty}(\mathbb{D})$ is bounded, i.e.,
\begin{equation}\label{weak}
\left | \left\{z \in \mathbb{D}: |T_{\nu}f(z)| > \tau \right\} \right | \le \left ( \frac{C\|f\|_{L^p(\mathbb{D})}}{\tau}\right )^{q}, \quad \text{for } \tau > 0, \, f \in L^p(\mathbb{D}).
\end{equation}
for some constant $C > 0$.

\subsubsection*{\bf The case $1 < c_\nu < 2$.} 
Let $c_0$ and $t_0$ be the constants in Lemma \ref{l-estimate of real part of K}. 
For any $t\in(0,t_0)$, define a set
\begin{equation}\label{E_t}
E_t := \left\{ \rho e^{i\theta} \in \mathbb{D} : 1 - t \le \rho \le 1 - \frac{t}{2},\ \left|\theta\right| \le \frac{t}{20} \right\},
\end{equation}
and its indicator function
\begin{equation}\label{f_t}
	f_t(z) = \mathbf{1}_{E_t}(z).
\end{equation}
By \eqref{|E_t|}, there exist absolute constants $c_1$ and $c_2$ such that
\begin{equation}\label{||f||2}
	c_1 t^2 \le |E_t| \le c_2 t^2 \quad \text{and} \quad c_1^{1/p} t^{2/p} \le \|f_{t}\|_{L^p(\mathbb{D})} \le c_2^{1/p} t^{2/p}.
\end{equation}
Therefore, by Lemma \ref{l-estimate of real part of K}, for all $z, \lambda\in E_t$ and $r\in[0,1]$,
the integrand defining $T_\nu f_t(z)$ has nonnegative real part, and for $r\in[1-t,1]$ it satisfies the stated lower bound.
\begin{equation}\label{real part of K}
  \operatorname{Re}\left(\frac{1}{(1-z\overline\lambda)(1-rz\overline\lambda)}\right) \ge \frac{c_0}{2t^2}.
\end{equation}
Then, for each $z \in E_t$, by \eqref{||f||2} and \eqref{real part of K}, we have
\begin{equation}\label{lowerbounded1}
\begin{aligned}
  \left|T_{\nu}f_t(z)\right| 
  &= \left| \int_{\mathbb{D}} \frac{1}{1 - z\overline{\lambda}} \int_0^1 \frac{1}{1 - rz\overline{\lambda}} \,d\nu(r) \mathbf{1}_{E_t}(\lambda) \,dA(\lambda) \right|\\
  &\ge \left| \operatorname{Re} \int_{E_t} \int_{0}^{1} \frac{1}{(1 - z\overline{\lambda})(1 - rz\overline{\lambda})} \,d\nu(r) \,dA(\lambda) \right| \\
  &\ge \int_{E_t} \int_{1 - t}^1 \operatorname{Re} \frac{1}{(1 - z\overline{\lambda})(1 - rz\overline{\lambda})} \,d\nu(r) \,dA(\lambda)  \\
  &\ge \frac{c_0}{2t^2} \nu([1 - t, 1)) \left|E_t\right| \ge \frac{1}{2}c_0 c_1 \nu([1 - t, 1)).
\end{aligned}
\end{equation}
Let $\tau = \frac{1}{4}c_0 c_1\nu([1 - t,1)),$ then $E_t \subset \{z \in \mathbb{D} : \left|T_{\nu}f_t(z)\right| > \tau\}.$
Hence, \eqref{weak} implies that
\[
\left( \frac{C\left\|f_t\right\|_{L^p(\mathbb{D})}}{\tau} \right)^{\!q} \ge \left| \left\{ z \in \mathbb{D} : \left|T_{\nu}f_t(z)\right| > \tau \right\} \right| \ge \left|E_t\right| \ge c_1  t^2.
\]
Set
$\widetilde{C}_{\nu} = 4 c_0^{-1} c_1^{-(1+1/q)}c_2^{1/p} C.$
Then combining \eqref{lowerbounded1}, 
\[
\nu([1 - t, 1)) \le \widetilde{C}_{\nu} \, t^{2/p - 2/q} = \widetilde{C}_{\nu} \, t^{2 - 2/c_{\nu}}, \quad \text{for } 0 < t < t_0,
\]
If $t_0\le t \le 1$, $\nu([1 - t, 1)) \le \nu([0,1))$, and $t^{2 - 2/c_{\nu}} \ge t_0^{2 - 2/c_{\nu}}$.
Hence, we have
\[
\sup_{0 < t \le 1} \frac{\nu([1 - t, 1))}{t^{2 - 2/c_{\nu}}} < \infty.
\]
This implies that $\nu$ satisfies the $2/c'_\nu$-Carleson condition.

\subsubsection*{\bf The case $c_{\nu} = 2$.} 
In this case, we have
$$
\frac{1}{q}=\frac{1}{p}-\frac{1}{2}.
$$
Let the set $E_t$ and function $f_t$ be defined as \eqref{E_t} and \eqref{f_t}. 
By Lemma \ref{l-estimate of real part of K}, for all $z, \lambda\in E_t$ and $r\in[0,1]$,
the integrand defining $T_\nu f_t(z)$ has nonnegative real part, and
\begin{equation}\label{real part of K 2}
  \operatorname{Re} \frac{1}{(1 - z\overline{\lambda})(1 - rz\overline{\lambda})} \ge \frac{c_0}{t((1 - r) + t)}.
\end{equation}
Then, for each $z \in E_t$, by \eqref{||f||2}, \eqref{lowerbounded1} and \eqref{real part of K 2}, we have
\begin{equation}\label{lowerbounded2}
\begin{aligned}
  \left|T_{\nu}f_t(z)\right| 
  &\ge \int_{0}^{1} \frac{c_0}{t\left((1 - r) + t\right)} \,d\nu(r) \left|E_t\right|\\
  &\ge c_0 c_1 t \int_{0}^{1} \frac{1}{(1 - r) + t} \,d\nu(r).
\end{aligned}
\end{equation}
Let
\[
\tau = \frac{1}{2}c_0 c_1 t \int_{0}^{1} \frac{1}{(1 - r) + t} \,d\nu(r),
\]
then $E_t \subset \{z \in \mathbb{D} : \left|T_{\nu}f_t(z)\right| > \tau\}$.
Hence, \eqref{weak} implies that
\[
\left( \frac{C\left\|f_t\right\|_{L^p(\mathbb{D})}}{\tau} \right)^{\!q} \ge \left| \left\{ z \in \mathbb{D} : \left|T_{\nu}f_t(z)\right| > \tau \right\} \right| \ge \left|E_t\right| \ge c_1 t^2.
\]
Then combining \eqref{lowerbounded2}, 
\[
\int_{0}^{1} \frac{1}{(1 - r) + t} \,d\nu(r) \le \widetilde{C}_{\nu},
\]
where $\widetilde{C}_{\nu} = 2c_0^{-1} c_1^{-(1+1/q)} c_2^{1/p} C$.
The Monotone Convergence Theorem yields
\[
\int_{0}^{1} \frac{1}{1-r}\, d\nu(r)
= \lim_{t\to 0}\int_{0}^{1} \frac{1}{(1-r)+t}\, d\nu(r)
\le \widetilde{C}_{\nu} < \infty.
\]
This completes the implication $(1)\Rightarrow (3)$.

\subsection{The proof of Theorem \ref{weak bd}}
Observe that the implication $(1)\Rightarrow (3)$ in Theorem \ref{weak bd on the line} remains valid for $p=1$ and $q=c_\nu$, which completes the proof of the necessity in Theorem \ref{weak bd}.

\subsection{The proof of Theorem \ref{dual of weak 1 alpha}}\label{Thm:4}
Recall it suffices to prove the necessary part of Theorem \ref{dual of weak 1 alpha}.
\subsubsection*{\bf The case $1 < c_\nu < 2$.}
Assume that $T_\nu:L^{c_\nu'}(\D)\to\mathcal B$ is bounded. Then there exists a constant $C>0$ such that for all $f \in L^{c_{\nu}'}(\mathbb{D})$,
\begin{equation}\label{bloch-seminorm}
\sup_{z\in\D}(1-|z|^2)\,|(T_\nu f)'(z)|\le C\|f\|_{L^{c_{\nu}'}(\mathbb{D})}.
\end{equation}
For $0 < t < 1$, let the set $E_t$ and function $f_t$ be defined as \eqref{E_t} and \eqref{f_t}. 
Then $$\|f_t\|_{L^{c_\nu'}(\D)}=|E_t|^{1/c_\nu'},$$ and there exist absolute constants $c_1,c_2>0$ such that
$
c_1 t^2\le |E_t|\le c_2 t^2,
$
\text{hence}
\[
c_1^{1/c_\nu'}t^{2/c_\nu'}\le \|f_t\|_{L^{c_\nu'}(\D)}\le c_2^{1/c_\nu'}t^{2/c_\nu'}.
\]
Fix $0<t<t_0$ and choose $z=\sqrt{1-t}\in E_t$ so that $t=1-|z|^2$.
\begin{align*}
  |(T_{\nu}f_t)'(z)| 
  &= \left |\int_{\mathbb{D}} \partial_z K_{\nu}(z,\lambda) \mathbf{1}_{E_t}(\lambda)dA(\lambda)\right |\\
  &= \left | \int_{E_t} \int_{0}^{1}\left(\frac{\overline{\lambda}}{(1 - z\overline{\lambda})^2(1 - rz\overline{\lambda})}  + \frac{r\overline{\lambda}}{(1 - z\overline{\lambda})(1 - rz\overline{\lambda})^2}  \right)d\nu(r) dA(\lambda) \right |.
\end{align*}
By Lemma \ref{lem:new19}, for all $\lambda\in E_t$ and $r\in[0,1]$,
the integrand defining $(T_\nu f_t)'(z)$ has nonnegative real part, 
\begin{equation}\label{eq: real part of partial K}
  \operatorname{Re}\left(\frac{\overline\lambda}{(1-z\overline\lambda)^2(1-rz\overline\lambda)}\right) \ge 0, \quad \text{and } \operatorname{Re}\left(\frac{\overline\lambda}{(1-z\overline\lambda)(1-rz\overline\lambda)^2}\right) \ge 0
\end{equation}
and for $r\in[1-t,1]$, it satisfies the stated lower bound:
\[
\operatorname{Re}\left(\frac{\overline\lambda}{(1-z\overline\lambda)^2(1-rz\overline\lambda)}\right) \ge \frac{\widetilde{c_0}}{2t^3}.
\]
Therefore,
\begin{align*}
  |(T_\nu f_t)'(z)|
&\ge \int_{E_t}\int_{1-t}^1 \operatorname{Re}\left(\frac{\overline\lambda}{(1-z\overline\lambda)^2(1-rz\overline\lambda)}+\frac{r\overline\lambda}{(1-z\overline\lambda)(1-rz\overline\lambda)^2}\right)\,d\nu(r)\,dA(\lambda)\\
&\ge \int_{E_t}\int_{1-t}^1 \operatorname{Re}\left(\frac{\overline\lambda}{(1-z\overline\lambda)^2(1-rz\overline\lambda)}\right)\,d\nu(r)\,dA(\lambda)\\
&\ge \frac{\widetilde{c_0}}{2t^3}\,|E_t|\,\nu([1-t,1]),
\end{align*}
where $\widetilde{c_0}= 1/108$.
Multiplying by $1-|z|^2$ and using \eqref{bloch-seminorm} give
\[
\widetilde{c_0}\,\frac{|E_t|}{2t^2}\,\nu([1-t,1])
\le (1-|z|^2)|(T_\nu f_t)'(z)|
\le C\|f_t\|_{L^{c_\nu'}(\D)}
\le C\,c_2^{1/c_\nu'}\,t^{2/c_\nu'}.
\]
Since $|E_t|\ge c_1 t^2$, and let $\widetilde{C} = 2\widetilde{c_0}^{-1}c_1^{-1} c_2^{1/{c_{\nu}'}}C$, we obtain
\[
\nu([1-t,1])\le \widetilde{C}\,t^{2/c_\nu'}, \quad 0<t<t_0.
\]
For $t_0\le t\le1$, finiteness of $\nu$ implies $\nu([1-t,1))\le \nu([0,1))$, hence
\[
\sup_{0<t\le1}\frac{\nu([1-t,1))}{t^{2/c_\nu'}}<\infty,
\]
i.e., $\nu$ satisfies the $2/c_\nu'$-Carleson condition.

\smallskip
\subsubsection*{\bf The case $c_\nu=2$.}
Let $E_t$ and $f_t$ be as above. Then $\|f_t\|_2=|E_t|^{1/2}$ and $c_1 t^2\le |E_t|\le c_2 t^2$.
Fix $0<t<t_0$ and recall $z=\sqrt{1-t}\in E_t$ so that $t=1-|z|^2$.
By Lemma \ref{lem:new19}, for all $\lambda\in E_t$ and $r\in[0,1]$,
the integrand defining $(T_\nu f_t)'(z)$ has nonnegative real part \eqref{eq: real part of partial K}, and for $r\in[0,1]$, it satisfies the stated lower bound:
\[
\operatorname{Re}\left(\frac{\overline\lambda}{(1-z\overline\lambda)^2(1-rz\overline\lambda)}\right) \ge \frac{\widetilde{c_0}}{t^2((1 - r) + t) }.
\]
Similarly, $\|f_t\|_2 = |E_t|^{1/2}$ and there exist absolute constants $c_1,c_2 > 0$ such that $c_1 t^2 \le |E_t| \le c_2 t^2$, and 
$c_1^{1/2} t\le \|f_t\|_2\le c_2^{1/2} t$, then
\begin{align*}
|(T_{\nu}f_t)'(z)| 
  &\ge \int_{E_t} \int_{0}^{1}\operatorname{Re}\left(\frac{\overline{\lambda}}{(1 - z\overline{\lambda})^2(1 - rz\overline{\lambda})}  + \frac{r\overline{\lambda}}{(1 - z\overline{\lambda})(1 - rz\overline{\lambda})^2}  \right)d\nu(r) dA(\lambda) \\
  &\ge  \int_{E_t} \int_{0}^{1} \operatorname{Re}\left(\frac{\overline{\lambda}}{(1 - z\overline{\lambda})^2(1 - rz\overline{\lambda})}\right)d\nu(r) dA(\lambda)\\
  &\ge |E_t| \widetilde{c_0} \int_{0}^{1} \frac{1}{t^2((1 - r) + t)}d\nu(r)\ge \widetilde{c_0} c_1 \int_{0}^{1} \frac{1}{(1 - r) + t}d\nu(r).
\end{align*}
Similarly, choose $z = (1 - t)^{1/2} \in E_t$ (so $t = 1 - |z|^2$), then
\[
t \widetilde{c_0} c_1 \int_{0}^{1} \frac{1}{(1 - r) + t}d\nu(r) \le (1 - |z|^2)|(T_{\nu}f_t)'(z)| \le C\|f_t\|_{L^{2}(\mathbb{D})} \le Cc_2^{1/2} t. 
\]
Dividing both sides by $t$ (since $t > 0$), and let $\widetilde{C} =  \widetilde{c_0}^{-1} c_1^{-1} c_2^{1/2}C$ we get
\[
\int_{0}^{1} \frac{1}{(1 - r) + t}d\nu(r) \le \widetilde{C}.
\]
Letting $t \to 0^+$ and applying the Monotone Convergence Theorem, we conclude
\[
\int_{0}^{1} \frac{1}{1 - r}d\nu(r) \le \widetilde{C} < \infty.
\]
This completes the proof.

%\bibliographystyle{plain}
%\bibliography{references}

\begin{thebibliography}{99}

\bibitem[AHR]{AHR}
Alexandru Aleman, H\r{a}kan Hedenmalm, and Stefan Richter.
\newblock Recent progress and open problems in the Bergman space.
\newblock \emph{Oper. Theory Adv. Appl.} \textbf{156} (2005), 27--59.

\bibitem[BKV]{BKV}
Stephen Buckley, Pekka Koskela, and Dragan Vukoti\'{c}.
\newblock Fractional integration, differentiation, and weighted Bergman spaces.
\newblock \emph{Math. Proc. Cambridge Philos. Soc.} \textbf{126} (1999), no.~2, 369--385.

\bibitem[CFWY]{CFWY}
Guozheng Cheng, Xiang Fang, Zipeng Wang, and Jiayang Yu.
\newblock The hyper-singular cousin of the Bergman projection.
\newblock \emph{Trans. Amer. Math. Soc.} \textbf{369} (2017), 8643--8662.

\bibitem[CRW]{CRW}
Ronald Coifman, Richard Rochberg, and Guido Weiss.
\newblock Factorization theorems for Hardy spaces in several variables.
\newblock \emph{Ann. of Math. (2)} \textbf{103} (1976), no.~3, 611--635.

\bibitem[DHZZ]{DHZZ}
Yaohua Deng, Li Huang, Tao Zhao, and Dechao Zheng.
\newblock Bergman projection and Bergman spaces.
\newblock \emph{J. Operator Theory} \textbf{46} (2001), no.~1, 3--24.

\bibitem[FGW]{FGW}
Xiang Fang, Kunyu Guo, and Zipeng Wang.
\newblock Composition operators on the Bergman space with quasiconformal symbols.
\newblock \emph{J. Geom. Anal.} \textbf{33} (2023), no.~4, Paper No. 125, 38~pp.

\bibitem[FR]{ForelliRudin1974}
Frank Forelli and Walter Rudin.
\newblock Projections on spaces of holomorphic functions in balls.
\newblock \emph{Indiana Univ. Math. J.} \textbf{24} (1974), 593--602.

\bibitem[Gra]{Gra}
Loukas Grafakos.
\newblock \emph{Modern Fourier Analysis}, 2nd ed.
\newblock GTM250, Springer, New York, 2009.

\bibitem[Hed]{Hed00}
H\r{a}kan Hedenmalm.
\newblock An off-diagonal estimate of Bergman kernels.
\newblock \emph{J. Math. Pures Appl. (9)} \textbf{79} (2000), no.~2, 163--172.

\bibitem[HKS]{HJS02}
H\r{a}kan Hedenmalm, Stefan Jakobsson, and Sergei Shimorin.
\newblock A biharmonic maximum principle for hyperbolic surfaces.
\newblock \emph{J. Reine Angew. Math.} \textbf{550} (2002), 25--75.

\bibitem[HS]{HS02}
H\r{a}kan Hedenmalm and Sergei Shimorin.
\newblock Hele-Shaw flow on hyperbolic surfaces.
\newblock \emph{J. Math. Pures Appl. (9)} \textbf{81} (2002), no.~3, 187--222.

\bibitem[HKZ]{HKZ}
H\r{a}kan Hedenmalm, Boris Korenblum, and Kehe Zhu.
\newblock \emph{Theory of Bergman Spaces}.
\newblock GTM199, Springer-Verlag, New York, 2000.

\bibitem[HZ25]{HZ25}
Bingyang Hu and Xiaojing Zhou.
\newblock From complex-analytic models to sparse domination: A dyadic approach of hypersingular operators via Bourgain's interpolation method.
\newblock \emph{arXiv:2512.24972} (2025).

\bibitem[KU]{Kaptanoglu2019}
H. Turgay Kaptano\u{g}lu and A. Ersin \"{U}reyen.
\newblock Singular integral operators with Bergman-Besov kernels on the ball.
\newblock \emph{Integral Equations Operator Theory} \textbf{91} (2019), no. 4, Paper No. 30, 30 pp.

\bibitem[Liu]{Liu2015}
Congwen Liu.
\newblock Sharp Forelli-Rudin estimates and the norm of the Bergman projection.
\newblock \emph{J. Funct. Anal.} \textbf{268} (2015), no.~2, 255--277.

\bibitem[Lue]{lue87}
Daniel Luecking.
\newblock Trace ideal criteria for Toeplitz operators.
\newblock \emph{J. Funct. Anal.} \textbf{73} (1987), no.~2, 345--368.

\bibitem[PRW]{PRW}
Jos\'{e} Pel\'{a}ez, Jouni R\"{a}tty\"{a}, and Brett Wick.
\newblock Bergman projection induced by kernel with integral representation.
\newblock \emph{J. Anal. Math.} \textbf{138} (2019), no.~1, 325--360.

\bibitem[Ru]{Ru}
Walter Rudin.
\newblock \emph{Real and Complex Analysis}, 3rd ed.
\newblock McGraw-Hil, New York, 1987.

\bibitem[Shi01]{Shi01}
Sergei Shimorin.
\newblock Wold-type decompositions and wandering subspaces for operators close to isometries.
\newblock \emph{J. Reine Angew. Math.} \textbf{531} (2001), 147--189.

\bibitem[Shi02]{S}
Sergei Shimorin.
\newblock An integral formula for weighted Bergman reproducing kernels.
\newblock \emph{Complex Var. Theory Appl.} \textbf{47} (2002), no.~11, 1015--1028.

\bibitem[Tao]{T.Tao}
Terry Tao.
\newblock \emph{Harmonic Analysis}.
\newblock Lecture notes, UCLA. Available at \url{https://www.math.ucla.edu/~tao/247a.1.06f/notes2.pdf}.

\bibitem[Zhao]{Zhao2015}
Ruhan Zhao.
\newblock Generalization of Schur's test and its application to a class of integral operators on the unit ball of $\mathbb{C}^n$.
\newblock \emph{Integral Equations Operator Theory} \textbf{83} (2015), no.~4, 519--532.

\bibitem[ZZ]{ZhaoZhou2022}
Ruhan Zhao and Lifang Zhou.
\newblock $L^p$-$L^q$ boundedness of Forelli-Rudin type operators on the unit ball of $\mathbb{C}^n$.
\newblock \emph{J. Funct. Anal.} \textbf{282} (2022), no. 5, Paper No. 109345, 26 pp.

\bibitem[Zhu14]{Zhu14}
Kehe Zhu.
\newblock \emph{Operator Theory in Function Spaces}, 2nd ed.
\newblock Mathematical Surveys and Monographs, vol. 138, AMS, 2014.

\bibitem[Zhu22]{Zhu2022}
Kehe Zhu.
\newblock Embedding and compact embedding of weighted Bergman spaces.
\newblock \emph{Illinois J. Math.} \textbf{66} (2022), no.~3, 435--448.


\end{thebibliography}

\end{document}